\documentclass[11pt,letterpaper]{amsart}
\usepackage{amsmath,amscd,amsfonts,amssymb}
\usepackage{color}
\usepackage{mathtools}
\usepackage{mathrsfs}
\usepackage[colorlinks=true]{hyperref}
\usepackage{a4wide}
\usepackage[textwidth=22mm,textsize=scriptsize, color=green!20]{todonotes}

\newtheorem{theorem}{Theorem}[section] 
\newtheorem{proposition}[theorem]{Proposition}
\newtheorem{lemma}[theorem]{Lemma}

\newtheorem{corollary}[theorem]{Corollary}

\newtheorem{thmA}{Theorem}

\theoremstyle{definition}
\newtheorem{definition}[theorem]{Definition}

\newtheorem{example}[theorem]{Example}

\newtheorem{remark}[theorem]{Remark} 
 
\newcommand{\dual}{\mathcal{D}}
\newcommand{\primal}{\mathcal{P}}

\newcommand{\A}{\mathbb{A}}
\newcommand{\C}{\mathbb{C}}
\newcommand{\HH}{\mathbb{H}}
\newcommand{\N}{\mathbb{N}}
\newcommand{\PP}{\mathbb{P}}
\newcommand{\Q}{\mathbb{Q}}
\newcommand{\R}{\mathbb{R}}

\newcommand{\Z}{\mathbb{Z}}
\newcommand{\cD}{\mathcal{D}}
\newcommand{\cL}{\mathcal{L}}
\newcommand{\csM}{\mathscr{M}}
\newcommand{\cM}{\mathcal{M}}

\newcommand{\cO}{\mathcal{O}}

\newcommand{\csP}{\mathscr{P}}
\newcommand{\cP}{\mathcal{P}}
\newcommand{\DR}{B_R}
\newcommand{\dd}{\mathrm{d}}

\newcommand{\musw}{\mu^{\mathrm{sw}}}

\newcommand{\loc}{\mathrm{loc}}
\newcommand{\weak}{\mathrm{weak}}

\newcommand{\FS}{\mathrm{FS}}

\newcommand{\transpose}{\mathrm{T}}
\newcommand{\op}{\mathrm{op}}
\newcommand{\hyp}{\mathrm{hyp}}
\newcommand{\Pet}{\mathrm{Pet}}

\newcommand{\Ht}{h}

\newcommand{\ginf}{\mathrm{g}}
\DeclareMathOperator{\capac}{cap}

\DeclareMathOperator{\Img}{Im}

\DeclareMathOperator{\ess}{ess}

\DeclareMathOperator{\supp}{supp}
\DeclareMathOperator{\essh}{\ess(\Ht_\ginf)}
\DeclareMathOperator{\Res}{Res}
\DeclareMathOperator{\rH}{H}

\begin{document}
%----------------------------------------
%\title[Approximating measures by integers and strong duality]{Approximating measures by algebraic integers and strong duality}
\title[Closing the gap around the essential minimum]{Closing the gap
  around the essential minimum of  height functions with linear
  programming} 
\author[Burgos Gil]{Jos{\'e} Ignacio Burgos Gil}
\address{Jos{\'e} Ignacio Burgos Gil:
Instituto de Ciencias Matem\'aticas (CSIC-UAM-UCM-UC3M), 
Calle Nicol\'as Cabrera 15, Campus UAM, Cantoblanco, 28049 Madrid, Spain}
\email{burgos@icmat.es}

\author[Menares]{Ricardo Menares}
\address{Ricardo Menares: Pontificia Universidad Católica de Chile, Facultad de Matemáticas, Vicuña Mackenna 4860, Santiago, Chile}
\email{rmenares.v@gmail.com}

\author[Qu]{Binggang Qu}
\address{Binggang Qu: Instituto de Ciencias Matem\'aticas (CSIC-UAM-UCM-UC3M), 
Calle Nicol\'as Cabrera 15, Campus UAM, Cantoblanco, 28049 Madrid, Spain}
\email{binggang.qu@icmat.es}

\author[Sombra]{Mart{\'\i}n~Sombra}
  \address{Mart\'in Sombra: 
Instituci\'o Catalana de Recerca
  i Estudis Avan\c{c}ats, Passeig Llu{\'\i}s Companys~23,
  08010 Barcelona, Spain;\newline
 \indent Departament de Matem\`atiques i
  Inform\`atica, Universitat de Barcelona, Gran Via 585, 08007
  Bar\-ce\-lo\-na, Spain;\newline
\indent Centre de Recerca Matemàtica, Edifici C, Campus Bellaterra, 08193 Bellaterra, Spain}
\email{martin.sombra@icrea.cat}

\date{\today} 
\thanks{Burgos Gil was partially supported by grants
  PID2022-142024NB-I00 and \mbox{CEX2023-001347-S} funded by
  MICIU/AEI/10.13039/501100011033. Menares was partially supported
  by ANID FONDECYT Regular grant 1250734. Qu was partially supported by
  grant PID2022-142024NB-I00 funded by
  MICIU/AEI/10.13039/501100011033.  Sombra was partially supported
  by  grants PID2023-147642NB-I00
  and CEX2020-001084-M funded by MICIU/AEI/10.13039/501100011033}

%----------------------------------------
\begin{abstract}
For many common height functions, it is notoriously hard to compute
the essential minimum. Nevertheless there are two classical methods,
one giving lower bounds and the other giving upper bounds. In this
paper, we show that the two methods are actually dual to each other in
the sense of linear programming. The main theorem is that they satisfy strong
duality, which closes the gap around the essential minimum from both
ends. As applications we prove that this essential minimum can
  be realized by a generic sequence of algebraic integers, and that if
  the associated Green function is computable then this essential
  minimum is a computable real number.
%	 For many common height functions, the explicit determination
%         of the  essential minimum is an open problem. We consider a
%         classical  method to obtain lower bounds that goes back at
%         least to C.J. Smyth, and a method  to obtain upper bounds
%         based on the knowledge of the limit distribution of integral
%         points. We use an infinite dimensional linear programming
%         scheme to show that both methods agree in the limit, by
%         showing that the principle of strong duality holds in our
%         situation. As a corollary we prove that the essential minimum
%         can be attained by sequences of algebraic integers.  
%	
%Recent results by A. Smith and B. Orloski--N. Sardari, furnish a
%characterization of compactly supported measures that can be
%approximated by complete sets of conjugates of algebraic integers, in
%terms of infinitely many nonnegativity conditions. We establish an
%extension of this characterization to measures with not necessarily
%compact support.
%As an application of this result and of strong duality, we show
%that the essential minimum is a computable
%real number when the Green function used to define the height is computable.
%We systematically use potential theory for measures that
%can integrate functions with logarithmic growth. 
\end{abstract}

\maketitle

%----------------------------------------
\setcounter{tocdepth}{1}
\tableofcontents

% ----------------------------------------

\section{Introduction}

\subsection{The essential minimum of height functions}
The height of an algebraic point in a quasi-projective variety is a
measure of its arithmetic complexity, which makes it an important tool
in the study of Diophantine equations. In particular it is crucial in
the proof of results like the Mordell--Weil theorem and the
Mordell--Faltings theorem, and appears in 
far-reaching conjectural statements such as the effective Mordell conjecture and
Vojta's conjecture, see for instance  \cite{BombieriGubler}.

Among the different properties of heights, the equidistribution of
small points is both of intrinsic interest and helpful in problems
about unlikely intersections such as the Manin--Mumford and the
Bogomolov conjectures. To formulate it, let $X/\mathbb{Q}$ be a
quasi-projective variety and
$h \colon X(\overline{\Q}) \longrightarrow \R$ a height function on
its set of algebraic points. The \emph{essential minimum} of $h$ is
the quantity defined as
\begin{displaymath}
  \ess(h) \coloneqq \inf \left\{ \liminf_{n \rightarrow \infty} h(\alpha_n) \, \middle| \, (\alpha_n)  \text{ is a generic sequence in }
    X(\overline{\Q}) \right\} \in \mathbb{R}\cup\{-\infty\},
\end{displaymath}
where a sequence of algebraic points of $X$ is called \emph{generic}
if it eventually escapes every proper Zariski closed subset. A generic
sequence $(x_n)$ in $X(\overline{\Q})$ is called \emph{small} if
$\lim_{n\to \infty}h(x_n) = \ess(h)$, namely if the height of these
points converges to the smallest possible value.  Then the
equidistribution of small points is the property that the Galois orbit
of the points in any small generic sequence converge towards a
prescribed measure.

The central result in this direction is Yuan's equidistribution
theorem \cite{YuanBig}, which most notably applies to the canonical
heights associated to dynamical systems, and in particular to the
canonical heights on toric varieties and the N\'eron--Tate heights on
abelian varieties. Beyond Yuan’s theorem, other cases that are
understood include toric heights on toric varieties~\cite{BPRStoric}
and canonical heights on semiabelian varieties \cite{Kuhne}.

Recently Balla\"y and the fourth author obtained a more general
equidistribution theorem unifying all these previous
results~\cite{BallaySombra}. However its implementation in any other
setting is challenging since in the first place it assumes the knowledge of
the essential minimum, which is a very
difficult problem.

For example consider the deceptively simple
\emph{Zhang--Zagier height}
$ h_{\mathrm{ZZ}} \colon \mathbb{A}^{1} (\overline{\Q})=
\overline{\Q} \longrightarrow \R $ defined by
\begin{displaymath}
h_{\mathrm{ZZ}} (\alpha)=  h_{\mathrm{W}}(\alpha) + h_{\mathrm{W}}(1-\alpha) \quad \text{ for } \alpha \in \overline{\mathbb{Q}},
\end{displaymath}
where $h_{\mathrm{W}}$ denotes the Weil height on $\overline{\Q}$.  Both lower and upper
bounds are known for its essential minimum
\cite{Zagier01, Doche1, Doche2} but the problem of
computing this quantity or even approximating it up to three
significant digits remains open.

\subsection{Height functions on $\overline{\Q}$}
In this paper we will study height functions on $\A^{1}(\overline{\Q})=\overline{\Q}$ whose
archimedean part is governed  by an arbitrary Green function. Precisely, a \emph{Green
  function} is a continuous function
$\ginf\colon \C \longrightarrow \R $ that is invariant under complex
conjugation and obeys the asymptotics
\begin{equation}\label{general asymptotic}
\ginf(z)=\log|z|+o(\log|z|) \quad \text{ as } |z|\longrightarrow \infty.
\end{equation}
The associated height function
$h_{\ginf} \colon \overline{\Q} \longrightarrow \R$ is then defined
for $ \alpha\in \overline{\Q}$ as
\begin{equation}\label{def height}
	h_{\ginf}(\alpha) = \frac{1}{\deg(\alpha)} \left( \log|c_{\alpha}| + \sum_{\beta \in O(\alpha)} \ginf(\beta) \right), 
\end{equation} 
where $\deg(\alpha) $ denotes the degree of this algebraic number,
$c_{\alpha}$ the leading coefficient of its minimal polynomial
$P_{\alpha}\in \mathbb{Z}[x]$ and $O(\alpha) \subseteq \mathbb{C}$ its
\emph{Galois orbit}, that in this situation might be defined as the
set of complex zeros of $P_{\alpha}$.  For instance, the choice
$\ginf(z) = \log^+|z|= \log \max(1,|z|)$ gives the Weil height on
$\overline{\mathbb{Q}}$.

 In the language of Arakelov geometry, these are the height functions
 corresponding to adelic metrics on the line bundle $\cO(1)$ on
 $\PP^1$ that for the archimedean place have mild singularities at the
 point at infinity whereas for the non-archimedean places are
 canonical, see Section \ref{Arakelov_theory_on_P^1} for details.

% The essential minimum of $h_\ginf$ is 
% \begin{equation}\label{defess}
% \ess(h_{\ginf})\coloneqq\inf \left\{ \liminf_{n \rightarrow \infty} h_{\ginf}(x_n) \,\middle|\, (x_n)\subseteq  \overline{\Q} \textrm{ is a  sequence of distinct algebraic numbers} \right\}.
% \end{equation}

 There are two classical methods for bounding the essential minimum of
 $h_\ginf$, one giving lower bounds and the other giving upper bounds.
 In this paper we show that both methods are actually dual to each
 other in the sense of linear programming.  Our main result (Theorem
 \ref{thmA:duality} below) shows that the strong duality property
 holds in this setting, closing the gap around the essential minimum
 from both ends. A surprising consequence is that for every such
 height function, the essential minimum can always be attained by a generic
 sequence of algebraic integers.

 \begin{thmA} \label{integral} Let
   $\ginf \colon \C \longrightarrow \R $ be a Green function. Then
   there exists a sequence $ (\alpha_n)$ of distinct algebraic
   integers such that $h_{\ginf}(\alpha_n)$ is monotonically
   decreasing and $\lim_{n} h_{\ginf}(\alpha_n) = \ess(h_{\ginf})$.
\end{thmA}

In fact we obtain that the set of height values
$\Ht_\ginf(\overline{\Z})$ is dense in the interval $[\essh,\infty)$
(Corollary~\ref{spectrum}), which in particular gives Theorem
\ref{integral}.

The Faltings height
$\Ht_{\operatorname{F}}\colon \overline{\Q}\longrightarrow \R$ is defined as the stable Faltings
height of the semistable elliptic curve with $j$-invariant equal to $\alpha$, for $\alpha\in \overline{\mathbb{Q}}$.  As
a specialization of the previous result, we obtain that the essential
minimum of Faltings height can be attained by a generic sequence of
$j$-invariants of different elliptic curves having good reduction
everywhere (Theorem \ref{Faltings_height_4}).

\subsection{Lower and upper
  bounds} \label{intro:subsection_bounds} Let
$P_{1},\dots ,P_{k}\in \Z[x]$ be nonzero polynomials with integer coefficients
and $a_{1},\dots, a_{k} \in \R_{\geq 0}$ non negative real numbers.
For every $\alpha\in \overline \Q$ such that $P_{i}(\alpha)\ne 0$ for all $i$ we have
\begin{displaymath}
  h_{\ginf}(\alpha) \geq \inf_{z \in \C} \left( \ginf(z)-\sum_{i=1}^{k}a_{i}\log \big| P_{i}(z) \big| \right).
\end{displaymath}
This inequality follows from the definition of the height and the
product formula on number fields and is at
the basis of Smyth's method \cite{Smyth:mtraiII}. It readily implies
\begin{displaymath}
\ess(h_\ginf)\geq \inf_{z \in \C} \left( \ginf(z)-\sum_{i=1}^{k}a_{i}\log \big| P_{i}(z) \big| \right)
\end{displaymath}
and so  the best lower bound   that one can obtain in this way is
\begin{displaymath}
\ess(h_\ginf) \geq  \sup \left\{ \inf_{z \in \C} \left( \ginf(z)-\sum_{i=1}^{k}a_{i}\log \big| P_{i}(z) \big| \right) \,\middle|\, k \in \mathbb{Z}_{\ge 0}, a_i\in \R_{\geq 0}, P_i\in \Z[x]\setminus \{0\}\right\}.  
\end{displaymath}

To describe the upper bounds, we denote by $\mathscr{P}_{\log}(\C )$ the space of
probability measures $\mu$ on $\C$~such~that
\begin{displaymath}
 \int\log^{+}|x| \, \dd \mu < \infty .
\end{displaymath}
A function $f\colon \C \longrightarrow \nobreak \R $ has
\emph{logarithmic growth} if there exist constants $A,B >0$ such that
$|f(z)| \leq A+B \log^+|z|$ for all $z \in \C$. We say that a
sequence  $(\mu_n)$ in~$ \mathscr{P}_{\log}(\C )$
converges \emph{$\log$-weakly} to $\mu\in \mathscr{P}_{\log}(\C )$ if
for every continuous function $f$ with logarithmic growth we have
\[
\lim_{n \rightarrow \infty} \int f \,\dd\mu_n = \int f \,\dd\mu.
\]
Then we denote by $\mathscr{P}^{\overline{\Z }}_{\log}(\C )$ the subset
of $\mathscr{P}_{\log}(\C )$ consisting of the limit points of
$\{ \delta_{O(\alpha)} \mid \alpha \in \overline{\Z}\}$ in the
$\log$-weak topology, where $\delta_{O(\alpha)}$ denotes the the
uniform probability measure on the Galois orbit of the algebraic
integer $\alpha$.

The basic observation is that for
$\mu \in \mathscr{P}^{\overline{\Z}}_{\log}(\C)$ and a sequence of
distinct algebraic integers $(\alpha_n)$ such that
$\delta_{O(\alpha_n)}$ converges $\log$-weakly to $\mu$ we have
\begin{equation}\label{eq. spectrum}
 \lim_{n\rightarrow \infty}h_{\ginf} (\alpha_n) = \int \ginf \,\dd\mu,
\end{equation}
because the height of an algebraic integer reduces to the average of
the Green function over its Galois orbit and the fact that Green
functions have logarithmic growth. Hence for every $\mu \in
\mathscr{P}^{\overline{\Z}}_{\log}(\C)$ we have 
\begin{displaymath}
  \ess(h_\ginf)\le \int \ginf \, \dd\mu.
\end{displaymath}
So, the best upper bound we can obtain this way is
\begin{displaymath}
  \ess(h_\ginf) \leq \inf \left\{ \int \ginf \, \dd\mu\,\middle|\, \mu \in \mathscr{P}^{\overline
    \Z}_{\log}(\C)\right\}.
\end{displaymath}
This is the principle applied by the first and second authors together
with Rivera-Letelier in their study of the Faltings height
\cite{BurgosGil_essMinFaltings}. To take advantage of this upper bound for the essential minimum it is
important to understand the probability measures that can be
arbitrarily approached by Galois orbits of algebraic integers.

\subsection{Approximation of measures by algebraic integers}
We now describe different ways to approximate measures by integers. 
 To this
end we first introduce some definitions and recall the previous
results in this direction.

We denote by $\mathscr{P}_{c}(\C)$ the space of probability measures
on $\C$ with compact support, and say that a sequence $(\mu_n)$ in
$\mathscr{P}_{c}(\C)$ \emph{converges properly} to
$\mu \in \mathscr{P}_{c}(\C)$ if it converges weakly  and there exists
a compact subset $K \subseteq \mathbb{C}$ 
containing the support of $\mu_{n}$ for all $n$.  Then we denote by
$\mathscr{P}^{\overline{\Z}}_{c}(\C) \subseteq \mathscr{P}_{c}(\C) $
the subset consisting of the limit points of
$\{ \delta_{O(\alpha)} \mid \alpha \in \overline{\Z} \}$ with respect
to the topology of the proper convergence.

Many measures in $\mathscr{P}^{\overline{\Z}}_{c}(\C)$ can be
constructed using potential theory. Let $\mu_K$ be the equilibrium
measure of a compact subset $K\subseteq \C$ that is invariant under
the complex conjugation and has capacity
${\operatorname{cap}(K)=1}$. Combining the Fekete--Szeg\H{o}
theorem~\cite{Fekete-Szego} with a result of Rumely \cite{Rumely} it
can be shown that $\mu_K \in \mathscr{P}^{\overline{\Z}}_{c}(\C)$
(Remark \ref{Rumely_equidistribution}).

A recent groundbreaking result by Smith \cite[Theorem 1.5]{Smith} and
by Orloski and Sardari~\cite[Theorem 1.2]{Orloski-Sardari} gives the
next elegant characterization: for $\mu \in \mathscr{P}_c(\C )$ we
have that $\mu \in \mathscr{P}^{\overline{\Z}}_c(\C )$ if and only if
$\mu$ is invariant under the complex conjugation and
\begin{equation}
  \label{eq:18}
\int \log|Q| \,\dd\mu \geq 0 \quad  \textrm{ for all } Q \in \Z [x]\setminus\{0\}.
\end{equation}
In particular $\mathscr{P}^{\overline{\Z}}_c(\C )$ is a convex subset
of $\mathscr{P}_c(\C )$. 

Actually both the Fekete--Szeg\H{o}--Rumely
result and the~Smith--Orloski--Sardari theorem are stronger (Remark
\ref{SOS-FS}) but the current versions suffice for our purposes.

% This characterization implies that
% $\mathscr{P}^{\overline{\Z}}_c(\C )$ is a convex set and that
% $\mu_K \in \mathscr{P}^{\overline{\Z}}_c(\C )$ for every compact
% conjugation invariant subset $K$ with $\operatorname{cap}(K)\geq
% 1$. On the other hand, Fekete \cite{Fekete} showed that whenever $K$
% has $\operatorname{cap}(K)<1$, there is an open neighborhood $U$ of
% $K$ such that only finitely many complete sets of conjugates of
% algebraic integers lie inside $U$. In particular,
% $\mu_K \notin \mathscr{P}^{\overline{\Z}}_{c}(\C) $ for such $K$.

% \begin{theorem}[Smith, Orloski--Sardari] \label{SOS_equidistribution}
% 	Assume $\mu \in \mathscr{P}_c(\C )$ is invariant under complex conjugation. Then, $\mu \in \mathscr{P}^{\overline{\Z}}_c(\C )$ if and only if $$\int \log|Q| \,\dd\mu \geq 0, \textrm{ for any } Q \in \Z [x].$$
% \end{theorem}

The use of log-weak convergence allows us to extend the~Smith--Orloski--Sardari theorem
to probability measures whose support is not necessarily compact.
\begin{thmA}[Theorem \ref{SOS non compact}] \label{SOS non compact
    introduction}
  Let $\mu \in \mathscr{P}_{\log}(\C )$. Then
  $\mu \in \mathscr{P}^{\overline{\Z}}_{\log}(\C )$ if and only if
  $\mu$ is invariant under the complex conjugation and
        \begin{equation*}
        \int \log|Q| \,\dd\mu \geq
                  0 \quad  \textrm{ for all } Q \in \Z [x]\setminus \{0\}.
\end{equation*}
\end{thmA}

% \begin{thmA}[Theorem \ref{SOS non compact}] \label{SOS non compact
%     introduction}
%   Let $\mu \in \mathscr{P}_c(\C )$. The following conditions are equivalent:
%   \begin{enumerate}
%   \item  $\mu \in \mathscr{P}^{\overline{\Z}}_{\log}(\C )$,
%   \item $\mu$ is invariant under the complex conjugation and
%     $\int \log|Q| \,\dd\mu \geq 0$ for all $ Q \in \Z [x]$.
%   \end{enumerate}
% \end{thmA}

Even though this
characterization may seem daunting because it involves infinitely many
inequalities, it allows us to obtain structural information on $\mathscr{P}^{\overline{\Z }}_{\log}(\C )$. For instance, Theorem \ref{SOS non compact
    introduction} readily implies that this is a convex set. Also, we can use it to construct a countable dense subset that allows a
complementary approach to  $\mathscr{P}^{\overline{\Z }}_{\log}(\C )$. 
%This result implies that $\mathscr{P}^{\overline{\Z }}_{\log}(\C )$ is
%a convex subset of $\mathscr{P}_{\log}(\C )$. 
%Even though its
%characterization may seem daunting because it involves infinitely many
%inequalities, we can construct a countable dense subset that allows a
%complementary approach to it. 
To this end, for each pair
$P, Q\in \Z[x]$ of different monic irreducible polynomials we denote
by $\mu_{P,Q}$ the % probability measure such
% that
% \begin{displaymath}
%   U^{\mu_{P,Q}}(z)=-\log \max\left\{\frac{|P(z)|}{\deg
%     P},\frac{|Q(z)|}{{\deg Q} +1}\right\}.
% \end{displaymath}
% is nothing but the 
normalized pullback of the Haar measure on the unit
circle under the rational map
\begin{displaymath}
\frac{P^{\deg (Q) + 1}}{Q^{\deg (P)}}\colon \mathbb{P}^1 \longrightarrow
\mathbb{P}^1.  
\end{displaymath}
Its support is the lemniscate
$\{ z\in \C \, | \, |P(z)|^{\deg (Q)+1}=|Q(z)|^{\deg (P)}\}$,
which is compact because the above rational function has a pole  at infinity. We show that $\mu_{P,Q}$ lies in
$\mathscr{P}^{\overline \Z}_{\log}(\C)$ and that the set of all such
measures is dense in $\mathscr{P}^{\overline \Z}_{\log}(\C)$ with
respect to the $\log$-weak topology (Theorem
\ref{mu_PQ_construction}).
% \begin{proposition} The family of probability measures 
%   $\{ \mu_{P,Q} \, |\ (P,Q)\in \Theta\}$ is a countable dense subset
%   of $\mathscr{P}^{\overline{\Z}}_{\log}(\C )$.
% \end{proposition}

 %This gives a partial answer to the
%analogue in our context of a question of Sarnak alluded to in
%\cite[Section 1.1]{OSFekete}.

\subsection{Potential theory}
\label{sec:potential-theory}
Potential theory plays an important role in the previous results.  Recall that
the \emph{potential} of a probability measure
$\mu\in \mathscr{P}_{\log}(\C)$ is the function
$U^{\mu}\colon \mathbb{C}\longrightarrow \mathbb{R}\cup \{\infty\}$
defined as
\begin{displaymath}
	U^{\mu} (z)= \int \log\frac{1}{|z-w|}\, \dd\mu(w).
\end{displaymath}
As a matter of fact, $\mathscr{P}_{\log}(\C)$ is the largest space of
probability measures with a well-defined
potential~\cite{BLW}. Furthermore a sequence $(\mu_{n})$ in
$\mathscr{P}_{\log}(\C)$ converges to $ \mu$ $\log$-weakly if and only
if the corresponding sequence of potential functions $ (U^{\mu_n})$
converges to $ U^\mu$ as distributions (Lemma \ref{superharmonic potential}, Lemma 
\ref{Hormander_psh_magic} and Proposition
\ref{log_convergence_IFF_potential_L^1_loc}). Hence
arguably the $\log$-weak convergence is the right topology on
$\mathscr{P}_{\log}(\C)$ to do potential theory.

We can characterize the $\log$-weak closure of
the set of equilibrium measures of compact sets with capacity one as
the set of measures with negative potential.

\begin{thmA}[Theorem \ref{negativo}] \label{negativo introduction} Let
  $\mu \in \mathscr{P}_{\log}(\C )$. Then $\mu$ is the $\log$-weak
  limit of a sequence of equilibrium measures $(\mu_{K_n})$ of
  compact subsets $K_n \subseteq \mathbb{C}$ with
  $\operatorname{cap}(K_{n})=1$ if and only if~$ U^{\mu}(z) \leq 0$
  for all $z\in \mathbb{C}$.  In particular, if $\mu$ is invariant
  under the complex conjugation and~$U^{\mu}(z) \le 0$ for all
  $z \in \mathbb{C}$ then
  $\mu\in \mathscr{P}^{\overline{\Z }}_{\log}(\C )$.
\end{thmA}

% \begin{thmA}[Theorem \ref{negativo}] \label{negativo introduction} Let
%   $\mu \in \mathscr{P}_{\log}(\C )$. The following conditions are equivalent:
%   \begin{enumerate}
%   \item $\mu$ is the $\log$-weak
%   limit of a sequence of equilibrium measures $(\mu_{K_n})_{n}$ of
%   compact subsets $K_n \subseteq \mathbb{C}$ with
%   $\operatorname{cap}(K_{n})=1$,
% \item $ U^{\mu}(z) \leq 0$ for all $z\in \mathbb{C}$.
%   \end{enumerate}
%   In particular, if $\mu$ is invariant under the complex conjugation
%   and~$U^{\mu}(z) \le 0$ for all $z \in \mathbb{C}$ then
%   $\mu\in \mathscr{P}^{\overline{\Z }}_{\log}(\C )$.
% \end{thmA}

Using this result it is easy to exhibit measures in 
$ \mathscr{P}^{\overline{\Z }}_{\log}(\C )$ that are not compactly
supported. A first example is the probability measure $\mu_{\FS}$
induced by the Fubini--Study form
\[
\omega_{\FS} = \frac{i\, \dd z  \wedge \dd \bar{z}}{2 \pi\, (|z|^2+1)^2},
\]
whose potential is the function
$U^{\mu_{\FS}}(z)= (-1/2) \log(1+|z|^2)\le 0$.

As a byproduct, this also allows to check that the set of equilibrium
measures of compact sets of capacity one are not dense in
$\mathscr{P}^{\overline \Z}_{\log}(\C)$. Indeed let $P_0=x-1$ and
$Q_0=x^2-x+1$. Applying Proposition \ref{prop:6} we find
\begin{equation}\label{exampleintro}
 U^{\mu_{P_0,Q_0}}\Big(\frac12\Big)=\frac13 \log\Big(\frac43\Big) >0 ,
\end{equation}
and so by Theorem \ref{negativo introduction} this measure cannot be
$\log$-weakly approximated by equilibrium measures of compact subsets
of capacity one.

\subsection{The essential minimum as a linear programming problem}
Let $\ginf\colon \mathbb{C}\longrightarrow \mathbb{R}$ be a Green function and set
\begin{align*}
  \dual(\ginf) & = \sup \left\{ \inf_{z \in \C} \left( \ginf(z)-\sum_{i=1}^{k}a_{i}\log \big| P_{i}(z) \big| \right) \,\middle|\, k \in \mathbb{Z}_{\ge 0}, a_i\in \R_{\geq 0}, P_i\in \Z[x]\setminus\{0\}\right\} ,\\
  \primal(\ginf) & = \inf \left\{ \int \ginf \, \dd\mu \,\middle|\, \mu\in \mathscr{P}_{\log}(\mathbb{C}) \text{ conjugation invariant and }
                   \int \log|Q| \,\mathrm{d}\mu \geq 0 \text{ for all $Q \in \Z[x]\setminus\{0\}$} \right\}.
\end{align*}
Combining the lower and upper bounds in
Section~\ref{intro:subsection_bounds} and the characterization of
$\mathscr{P}^{\overline{\Z }}_{\log}(\C)$ in Theorem \ref{SOS non
  compact introduction} we obtain
\begin{equation}\label{eq:19}
\dual(\ginf) \leq \ess(h_\ginf) \le \primal(\ginf).
\end{equation}
A key observation is that:
\begin{enumerate}
\item $\primal(\ginf)$ is a linear programming
  problem, in the sense that the objective function and the
  constraints are all linear functionals on $\mu$,
	\item $\dual(\ginf)$ is the dual problem of $\primal(\ginf)$ in the sense of linear programming.
\end{enumerate}
These problems are akin to those considered by Smyth
\cite{Smyth:tpaist}, Smith \cite{Smith} and Orloski, Sardari and Smith
\cite{Smith-Orloski-Sardari} in their works on totally real algebraic
integers of small trace, which were indeed a major source of
inspiration for us.

The previous observation can be made rigorous (Appendix
\ref{sec:dual-line-optim}), which gives the weak duality property
$\dual(\ginf) \leq \primal(\ginf)$ in accordance to \eqref{eq:19}.
Our main result is the following strong duality theorem showing the
equality between these terms.

\begin{thmA}[Theorem \ref{strong_duality_essential_minimum}] \label{thmA:duality}
	We have $\dual(\ginf) = \ess(h_\ginf) = \primal(\ginf)$.
\end{thmA}

% As mentionned before, this result implies that the set
% $\Ht_\ginf(\overline{\Z})$ is dense in the interval $[\essh,\infty)$.

% , and in particular that the essential minimum of $h_{\ginf}$ can be
% arbitrarily approached by heights of algebraic integers. which in
% particular gives Theorem \ref{integral}.
Assuming $\ginf$ is subharmonic and $\ginf(z)=\log|z|+a+o(1)$ as
$z\longrightarrow \infty$, the equality $\dual(\ginf)=\ess(\ginf)$ is
a particular case of a theorem of Balla\"y \cite[Theorem
1.2]{Ballay}, see Remark \ref{rem:5}.

\subsection{The computability of the essential minimum}
\label{sec:comp-essent-minim-1}

Our closing point concerns the computability of the essential
minimum. To place this problem into context, we first recall the known
approximations for the specific heights we mentioned before.

In \cite{Zagier01} Zagier obtained a lower bound for the
essential minimum of the Zhang--Zagier height which  was later improved by
Doche \cite{Doche1, Doche2}. Doche  also obtained an upper bound, giving 
\begin{equation*} %\label{ZZ}
0.248247  \leq \ess(\Ht_{\mathrm{ZZ}})\leq  0.254437.
\end{equation*}
In \cite{Lobrich} L\"obrich obtained a lower bound for the essential
minimum of the Faltings height, and both lower and upper bounds
were obtained in \cite{BurgosGil_essMinFaltings}. This gave
\begin{equation*} %\label{FF}
-0.748629 \leq  \ess(\Ht_{\mathrm{F}})\leq  -0.748622.
\end{equation*}

The above are the best known bounds for these essential minima, and
actually no practical algorithm is known to approximate them up to any
arbitrary precision.  Indeed, after a few iterations the methods that
produce these bounds reach a point where it is unclear how to
continue, due to the enormous size of the search space and the lack of
an efficient criterion to find the optimal direction, see for instance
the discussions in \cite[Section 3]{Zagier01} and \cite[Section
5]{Doche1} for the Zhang--Zagier height and
\cite[Section~8]{BurgosGil_essMinFaltings} for the Faltings height.

As an application of our results we show that these methods do
converge towards these essential minima.  Indeed, for each of these
heights we can reduce both $\dual(\ginf)$ and $\primal(\ginf)$ to
countable subsets (Propositions \ref{prop:5} and Theorem
\ref{mu_PQ_construction}) allowing to set up a theoretical algorithm
that produces  sequences of lower and upper bounds  converging to the same limit.  When these
bounds are closer than the given precision, it has found the required
approximation and~stops.

Loosely speaking, a real number $\lambda$ is \emph{computable} if
there exists an algorithm that is able to compute a rational
approximation of $\lambda$ up to any given precision. We also say that
a Green function $\ginf\colon \mathbb{C}\longrightarrow \mathbb{R}$ is
\emph{computable} if the asymptotics of $\ginf$ can be effectively
controlled and there exists an algorithm that can compute rational
approximations of $\ginf$ of arbitrary precision within any bounded
rectangle, see Section \ref{sec:comp-essent-minim} for the formal
definitions.

The theoretical  algorithm described before can be extended to this
more general situation and leads to the following result.

\begin{thmA}[Theorem \ref{ess_min_computable}] \label{computability}
  Let $\ginf \colon \C\longrightarrow \R$ be a computable Green function. Then $\ess(\Ht_{\ginf})$ is a computable real number.  
\end{thmA}

Ultimately we would like to have an practical algorithm to compute the
essential minimum of specific height functions like those of
Zhang--Zagier and Faltings.  The fact that it can be reached from both sides through linear
programming might suggest strategies allowing to produce numerical
approximations with arbitrary precision.

\subsection*{Acknowledgments}
We thank Fran\c cois Balla\"y, Norman Levenberg, Mayuresh Londhe,
Jo\-a\-quim Ortega Cerd\`a, Crist\'obal Rojas and Micha\l{}
Szachniewicz for illustrating discussions, and specially Nuno Hultberg
for sharing with us the example in Section \ref{Nuno}. We also thank
the organizers of the conference ``Arithmetic and algebraic geometry
week" held at Universitatea Alexandru Ioan Cuza in Ia\c{s}i on
September 2025, where part of this work was~done.

%%% Local Variables:
%%% mode: latex
%%% TeX-master: "main"
%%% End:

\section{Potential theory and limit distribution of algebraic integers}
\label{sec:preliminaries}
\subsection{Potential theory on the complex plane and the Fekete--Szeg\H{o}
  theorem}
\label{sec:potent-theory-compl}
Let $\mathscr{P}(\C )$ be the set of all probability Borel
measures on $\C $ and let $\mathscr{P}_{\log}(\C )$ be
the subspace of those that can integrate $\log^+|z|$, i.e. 
\[
\mathscr{P}_{\log}(\C ) \coloneqq \left\{ \mu \in \mathscr{P}(\C ) \,\middle|\, \int \log^+|z| \,\dd\mu < \infty \right\}.
\]
Let $\mathscr{P}_c(\C )$ be the set of compactly supported
measures. Then $\mathscr{P}_c(\C ) \subseteq
\mathscr{P}_{\log}(\C )$. 

Potential theory of measures with non-necessarily compact support
has been developed for some time, see for example
\cite{BLW},\cite{OSLW}.  We recall the basic facts. 
For any $\mu \in \mathscr{P}_{\log}(\C )$, we define its \emph{potential function} as
\[
U^\mu \colon \C \longrightarrow \R\cup\{+\infty\}, \quad U^\mu(z)
\coloneqq \int \log\frac{1}{|w-z|} \,\dd\mu(w).
\]

\noindent The value $U^\mu(z)$ is well-defined in $\R\cup\{+\infty\}$ because the negative part $$\int_{|w-z|\geq 1} -\log |w-z| \,\dd \mu(w)$$ is finite.

\begin{lemma}\label{superharmonic potential}
Let $\mu \in  \mathscr{P}_{\log}(\C )$. Then $-U^{\mu}$ is subharmonic and hence locally integrable.
\end{lemma}
\begin{proof}
  The first statement is \cite[Lemma 3.2]{BLW} and the second is a
  classical result on subharmonic functions. 
% Let $v(z,w)\coloneqq \log\left(\frac{|z-w|}{1+|w|}\right)$. Since $\mu\in  \mathscr{P}_{\log}(\C )$, we have that 
% \begin{equation}\label{integrand}
% -U^{\mu}(z)-\int \log(1+|w|)\textrm{d} \mu(w)=\int v(z,w)\textrm{d} \mu(w).
% \end{equation}
% For all $w$, the function $z \mapsto v(z,w)$ is subharmonic at $w$ and
% harmonic elsewhere. Also, the function $z \mapsto \sup \{v(z,w) \mid
% w\in \C\}$ is locally bounded above. Hence, the  right hand side of
% \eqref{integrand} is a subharmonic function of $z$ by the criterion in
% \cite[Theorem 2.4.8]{Ransford}.  
\end{proof}

To be on the safe side, we next check that Fubini--Tonelli theorem can be
applied  to potentials of measures in $\csP_{\log}(\C)$.
\begin{lemma}\label{Fubini}
Let $\mu_{1}, \mu _{2}  \in  \mathscr{P}_{\log}(\C )$. Then
\begin{displaymath}
  \iint  \log |w-z|^{-1}\,\dd \mu_{1}(w) \dd\mu _{2}(z) = \iint
  \log |w-z|^{-1}\,\dd \mu_{2}(z) \dd\mu _{1}(w).
\end{displaymath}
In other words, if $U_{1}$ and $U_{2}$ are the potentials of $\mu
_{1}$ and $\mu _{2}$, then
\begin{displaymath}
 \int  U_{1}\dd\mu _{2} = \int
 U_{2} \dd\mu _{1}. 
\end{displaymath}
\end{lemma}
\begin{proof}
  We write $\log^{-}|z|=-\min\{0,\log|z|\}$ so that
  $\log|z|=\log^{+}|z|-\log^{-}|z|$. We first check that 
\begin{equation}\label{eq:20}
  \iint  \log^{+} |w-z|\,\dd \mu_{1}(w) \dd\mu _{2}(z) = \iint
  \log^{+} |w-z|\,\dd \mu_{2}(z) \dd\mu _{1}(w) < \infty. 
\end{equation}
Indeed, using that $\mu _{1}  \in  \mathscr{P}_{\log}(\C )$ we have that
\begin{displaymath}
  \int  \log^{+} |w-z|\,\dd \mu_{1}(w)\le
  \int \log(2) + \log^{+}|w|+\log^{+}|z| \,\dd \mu_{1}(w)\le A+\log^{+}|z|
\end{displaymath}
for some real number $A>0$. Since $\mu _{2}  \in
\mathscr{P}_{\log}(\C )$, we deduce
\begin{displaymath}
    \iint  \log^{+} |w-z|\,\dd \mu_{1}(w) \dd\mu _{2}(z) <\infty
  \end{displaymath}
  and equation \eqref{eq:20} follows from the classical Fubini--Tonelli
  theorem.
  We turn to $\log^{-} |w-z|$. If one of the integrals
  \begin{displaymath}
  \iint  \log^{-} |w-z|\,\dd \mu_{1}(w) \dd\mu _{2}(z), \; \iint
  \log^{-} |w-z|\,\dd \mu_{2}(z) \dd\mu _{1}(w) 
  \end{displaymath}
is finite then by the Fubini--Tonelli theorem both integrals are
finite and agree. If both are $\infty$ then of course they also agree.
Since the positive part is finite, we conclude.
  % It is enough to check that
%   \begin{equation}\label{eq:5}
%     \iint
%   \left|\varphi(z) \log |w-z|^{-1}\right|\dd\lambda(z) \dd \mu(w)< +\infty
% \end{equation}
% Since $\log|z|$ is locally integrable, the function $\int
%   |\varphi(z) \log |w-z|^{-1}|\dd\lambda(z)$ is a continuous function
%   on $\C$. Clearly it has logarithmic growth. Since $\mu \in
%   \csP_{\log}(\C)$ the condition \eqref{eq:5} is satisfied.  
% For $M,r\in \R_{>0}$, set $\varphi_M(z,w)=\min\{M,\varphi(z)\log |w-z|^{-1}\}$. 
% Applying the monotone convergence theorem, we have that 
% \begin{displaymath}
%   \iint \varphi(z)\log |w-z|^{-1}  \dd \mu(w) \dd\lambda(z) = \lim_{M
%     \rightarrow \infty} \iint \varphi_M(z,w)\dd\mu(w)
%   \dd\lambda(z). 
% \end{displaymath}
% \todo{jose: I do not see this proof. ricardo: how about now?}
% We claim that $\varphi_M$  is
% integrable with respect  to $\mu \otimes \lambda$. Indeed, there
% exists $R>1$ such that $\varphi_M(z,w) $ is nonzero only if $|z|\leq R$. For
% such $z$, we have
% \begin{displaymath}
%   \sup |\varphi| \log (R+|w|)^{-1}\leq \varphi_M(z,w)\leq
%   M,
% \end{displaymath} so $|\varphi_M(z,w)|\leq \max\{M, \sup |\varphi| \log (R+|w|)\}$. 
% Since $\mu  \in  \mathscr{P}_{\log}(\C )$, this estimate justifies our
% claim. Hence, by Fubini's theorem
% \begin{displaymath}
%   \iint  \varphi_M(z,w)\dd\mu(w) \dd\lambda(z)= \iint  \varphi_M(z,w) \dd\lambda(z)\dd\mu(w) .
% \end{displaymath}
% We conclude by another application of the monotone convergence theorem.
\end{proof}

\begin{lemma} \label{true_Fubini}
	Let $\mu \in \mathscr{P}_{\log}(\mathbb{C})$ and $\varphi \in C_c^\infty(\mathbb{C})$. Then
	\begin{displaymath}
		\iint \varphi(w) \log|z-w|^{-1} \,\mathrm{d}\mu(z) \mathrm{d} \lambda(w) = \iint \varphi(w) \log|z-w|^{-1} \,\mathrm{d}\lambda(w) \mathrm{d}\mu(z),
	\end{displaymath}
	where $\lambda$ is the Lebesgue measure on $\mathbb{C}$.
\end{lemma}

\begin{proof}
	It follows from the fact that
        \begin{displaymath}
        z \longmapsto \int \varphi(w)\log|z-w|^{-1} \dd\lambda(w) 
        \end{displaymath}
        is a continuous function of logarithmic growth, hence it is absolutely integrable by $\mu \in \mathscr{P}_{\log}(\mathbb{C})$ and the Fubini--Tonelli theorem applies.
\end{proof}

As in the case of compactly supported measures, we recover the measure
from its potential. We denote by $\Delta $ the distributional
Laplacian. 
\begin{lemma}\label{lemm:1}
  Let $\mu  \in  \mathscr{P}_{\log}(\C )$. Then $-\Delta U^{\mu }=2\pi
  \mu $ as distributions. 
\end{lemma}
\begin{proof}
  Let $\varphi \in C_c^\infty(\C )$ be a smooth function with compact support. For smooth functions we abuse
  notation and write $\Delta \varphi=\Delta (\varphi)\lambda $,
  where $\Delta(\varphi)$ is a the Laplacian of $\varphi$ as a function and $\lambda $ is the Lebesgue measure in $\C$. Then
  \begin{align*}
    (-\Delta U^{\mu})(\varphi)
    &=\int -U^{\mu }(z)\Delta (\varphi)(z)\,\dd \lambda(z) \\
    &=\iint \log|z-w| \Delta (\varphi)(z) \,\dd \mu (w) \dd \lambda (z)\\
    &=\iint \log|z-w| \Delta (\varphi)(z) \,\dd \lambda (z) \dd \mu (w)\\
    &=\int 2\pi \varphi(w) \dd \mu (w)\\
    &=2\pi \mu (\varphi).
  \end{align*}
  Here the third equality is Lemma \ref{true_Fubini}, and the fourth equality
  is   a classical application of Green's formula (e.g. as in the proof of \cite[Theorem 3.7.4]{Ransford})
  \begin{displaymath}
    \int \log|z-w| \,\dd(\Delta \varphi)(z)=2\pi \varphi(w).
  \end{displaymath}  
\end{proof}

We quote a result that will be useful later.
\begin{lemma} \label{Hormander_psh_magic}
	Let $f, f_n \colon \C  \longrightarrow \R \cup \{-\infty\}$, $n \geq 1$, be subharmonic functions. Then $f_n$ converges to $f$ in
        $L^1_{\loc}$ if and only if $f_n$ converges to
        $f$ as distributions.  
\end{lemma}
\begin{proof}
	\cite[Theorem 3.2.13]{Hormander}.
\end{proof}

%Note that $U^{\mu}(x) = -\log|x| + o(1)$ when $|x| \longrightarrow \infty$ \textcolor{blue}{I think this is not true (e.g. for $\mu=\sum_n \frac{1}{n^2}\delta_n$). I only know how to prove $U^\mu(x)\geq -\log |x|+O(1)$}. 

We next recall the definition of capacity and the Fekete--Szeg\H{o}
theorem.  
The \emph{energy} of a measure $\mu \in \mathscr{P}_{\log}(\C )$ as is defined as
\[
I(\mu) = \int U^{\mu} \,\dd\mu = \iint -\log|w-z|
\,\dd\mu(w)\dd\mu(z) \in \R\cup\{+\infty\}.
\]

Let $K \subseteq \C $ be a compact subset, and let $\mathscr{P}(K)$ be
the space of all probability measures supported on $K$. We say $K$ is
\emph{polar} if $I(\mu)=+\infty$ for all $\mu \in \mathscr{P}(K)$. 

Now assume $K$ is non-polar, then there exists a unique  measure
$\mu_K \in \mathscr{P}(K)$ such that $I(\mu_K) = \inf_{\mu \in
  \mathscr{P}(K)} I(\mu)$ \cite[Theorem 3.7.6]{Ransford}. This $\mu_K$
is called the \emph{equilibrium measure} of $K$, and the
\emph{capacity} of $K$ is defined as $\operatorname{cap}(K) \coloneqq
e^{-I(\mu_K)}$.
 
Let $\alpha $ be an algebraic number and $O(\alpha ) \coloneqq
\operatorname{Gal}(\overline{\Q }/\Q ) \cdot \alpha $ be the Galois
orbit of $\alpha $, that is the complete set of conjugates of $\alpha
$. Let $S \subseteq \C $ be a subset, we say $\alpha $ is \emph{totally in
  $S$} if $O(\alpha ) \subseteq S$. We denote by $\delta_{O(\alpha)}$ the associated discrete measure
 \begin{displaymath}
   \delta _{O(\alpha)}=\frac{1}{\deg(\alpha)}
   \sum_{\beta  \in O(\alpha)} \delta_\beta.
 \end{displaymath} 
\begin{theorem}[Fekete--Szeg\H{o} \cite{Fekete,Fekete-Szego}] \label{Fekete-Szego}
	Let $K \subseteq \C $ be a compact subset. Then
\begin{enumerate}
\item if $\operatorname{cap}(K)<1$, then there exists an open
  neighborhood $U$ of $K$, such that there are only finitely many
  algebraic integers totally in $U$. 
\item if $\operatorname{cap}(K) \geq 1$ and $K$ is invariant under
  complex conjugation, then for any open neighborhood $U$ of $K$,
  there are infinitely many algebraic integers totally in $U$. 
\end{enumerate}
\end{theorem}

\begin{remark} \label{Rumely_equidistribution}
	In the critical case $\operatorname{cap}(K) = 1$, let $(
        \alpha _n )$ be a sequence of algebraic integers such that
        $O(\alpha _n) \subseteq K_{1/n}\coloneqq \big\{ z \mid |z-w|<1/n \text{ for
          some }w \in K \big\}$. Then the sequence of measures $\delta _{O(\alpha _{n})}$
        converges to the equilibrium measure $\mu_K$
        \cite[Theorem 1]{Rumely}. 
\end{remark}

\subsection{Smith and Orloski--Sardari theorems}
\label{sec:sos-equid-theor}
When $\operatorname{cap}(K)\geq 1$, by Fekete--Szeg\H{o} there will
be infinitely many algebraic integers totally contained in any open
neighborhood of $K$ and we would like to know what are the possible accumulation measures describing their asymptotic distribution. For $\capac(K)=1$ this is answered by a result of
Rumely as in Remark \ref{Rumely_equidistribution}, while for $\capac(K)>1$, this is answered in the recent breakthrough by
Smith and Orloski--Sardari.

We say a sequence $( \mu_n )$ in $\mathscr{P}(\C )$ \emph{converges weakly} to $\mu$, if we have for any continuous and compactly supported function $f \colon \C  \longrightarrow \R $,
\[
\lim_{n \rightarrow \infty} \int f \,\dd\mu_n = \int f \,\dd\mu.
\]
We can replace ``compactly supported" by ``bounded" \cite[Lemma 2.2]{Bi97}.

For compactly supported measures there is a stronger notion of
convergence. We say a sequence $(
\mu_n )$ in $\mathscr{P}_c(\C )$ \emph{converges properly} to
$\mu \in \mathscr{P}_c(\C )$, if it converges weakly and there exists a compact subset $K \subseteq \C $ such that $\operatorname{supp} (\mu_n) \subseteq K$ for all $n$. Note that this is equivalent to
\begin{equation}\label{any continuous}
\lim_{n \rightarrow \infty} \int f \,\dd\mu_n = \int f \,\dd\mu
\end{equation} for any continuous function $f \colon \C  \longrightarrow \R $.

A reformulation of the result by Smith \cite[Theorem 1.5]{Smith} and
Orloski--Sardari \cite[Theorem 1.2]{Orloski-Sardari} is
\begin{theorem}[Smith, Orloski--Sardari]\label{thm:2}
	Assume $\mu \in \mathscr{P}_c(\C )$ is invariant under complex
        conjugation. Then the following are equivalent
	\begin{enumerate}
		\item\label{item:1} there exists a sequence of
                  distinct algebraic integers $(\alpha _n ) $ such that $\delta_{O(\alpha _n)}$ converges properly
                  to  $ \mu$,
		\item\label{item:2} $\displaystyle \int \log|Q|
                  \,\dd\mu \geq 0$ for any $Q \in \Z [x]$. 
	\end{enumerate}
\end{theorem}

\begin{remark} \label{degree_unbounded}
	Let $\Ht \colon \overline{\Q } \longrightarrow \R $ be
        the standard Weil height function. Suppose $(\alpha _n )$ is a
        sequence of distinct algebraic
        integers such that $\delta_{O(\alpha _n)}$ converges properly
        to $\mu \in \mathscr{P}_c(\mathbb{C})$, then 
	\[
	\lim_{n\to \infty}\Ht(\alpha _n) = \lim_{n\to \infty}\int \log^+|z|\,\dd \delta _{O(\alpha _n)}
        = \int \log^+|z|\,\dd\mu < \infty. 
      \]
      This forces $\deg(\alpha _n)$ to converge to $\infty$ because
        otherwise it would contradict the Northcott property. 
\end{remark}

\begin{remark}\label{SOS-FS}
	The original results of Smith and Orloski--Sardari are in fact
        stronger and are more of Fekete--Szeg\H{o} style, in the sense that
        they prove that for any $\varepsilon>0$ there exists a sequence of
        algebraic integers $(\alpha _n)$ that are totally contained in
	\[
	\supp(\mu)_{\varepsilon} \coloneqq \big\{ z\in \C \mid
        |z-w|<\varepsilon \text{ for some }w \in
          \supp(\mu) \big\}. 
      \]
      and such that $\delta_{O(\alpha _n)}$ converges weakly to  $\mu$. This is
      much stronger than just proper convergence.  
      Besides, when $\mu$ is supported on $\R $, they prove that 
      the approximating algebraic integers $\alpha _n$ can be chosen
      to be totally real.
\end{remark}

\subsection{$\log$-weak convergence of measures}
We say a function $f \colon \C  \longrightarrow \nobreak
\R $ has \emph{logarithmic growth}, if there exist
constants
$A,B \geq 0$ such that for all $z \in \C $, we have $|f(z)| \leq A+B
\log^+|z|$. 

\begin{definition} \label{def:log-tight} Let $\mu\in
  \mathscr{P}_{\log}(\C )$ and  $(\mu_n)$ be a sequence in
  $\mathscr{P}_{\log}(\C )$. We say $\mu_n$
\emph{converges $\log$-weakly} to $\mu$, if for any continuous
function $f$  with logarithmic growth, we have 
\[
\lim_{n \to \infty} \int f \,\dd\mu_n = \int f \,\dd\mu.
\]
\end{definition}

\begin{definition}\label{def:5}
  Let $\varphi \colon \C \longrightarrow \R_{\ge 0}$ be a continuous
  function. We say that a sequence $( \mu_n )$ is \emph{$\varphi$-tight}, if
  \[
    \lim_{R \rightarrow \infty} \sup_{n\in \N} \int_{|z| \geq R} \varphi
    \,\dd\mu_n = 0.
  \]
  In particular, we say that a sequence $( \mu_n )$ is \emph{tight} if it is
  $1$-tight, and we say that it is \emph{$\log$-tight},
  if it is $\log^+$-tight.
\end{definition}
The next two lemmas are elementary and follow easily from the definition. 
\begin{lemma} \label{tight_transition}
	Let $\varphi,\psi \colon \C \longrightarrow \R_{\ge 0}$ be continuous functions
        with $\psi =o(\varphi)$ as $|z| \longrightarrow \infty$. Suppose we have
        a sequence $(\mu_n ) $ such that 
	\[
	\sup_{n\in \N} \int \varphi \,\dd\mu_n < \infty.
      \]
      Then $(\mu_n )$ is $\psi $-tight.
\end{lemma}
\begin{lemma}\label{estrecha}
Let $\mu\in
  \mathscr{P}_{\log}(\C )$ and  $(\mu_n)$ be a sequence in
  $\mathscr{P}_{\log}(\C )$ that converges weakly to $\mu$. Then, the following assertions are equivalent:
\begin{enumerate}
\item the sequence $(\mu_n)$ converges $\log$-weakly to $\mu$;
\item the sequence $(\mu_n)$ is log-tight;
\item we have that 
  \begin{displaymath}
    \lim_{n\rightarrow \infty} \int \log^+|z|\,\dd\mu_n(z)=\int \log^+|z| \,\dd\mu(z).
  \end{displaymath}
\end{enumerate}
\end{lemma}

We record here the following fact for future use.

\begin{lemma}\label{monotone}
Let $(\mu_n ) $ be a sequence in $\mathscr{P}(\C )$
that converges $\log$-weakly to $\mu$. Then, for all polynomials
$Q(x) \in \C[x]$, we have that
\begin{displaymath}
  \int \log |Q|\,\dd\mu \geq \liminf_{n\to \infty}
  \int \log |Q|\,\dd\mu_n.
\end{displaymath}
\end{lemma}

\begin{proof}
Using the monotone convergence theorem and \eqref{any continuous}, we have that 
\begin{align*}
\int \log |Q|\,\dd\mu& = \lim_{M\rightarrow \infty} \int \max\{-M, \log |Q|\}\,\dd\mu\\
& = \lim_{M\rightarrow \infty} \lim_{n\rightarrow \infty}\int
  \max\{-M, \log |Q|\} \,\dd\mu_n\\
& \ge \lim_{M\rightarrow \infty} \liminf_{n\rightarrow \infty}\int
  \log |Q| \,\dd\mu_n, \\
  &= \liminf_{n\rightarrow \infty}\int
  \log |Q| \,\dd\mu_n
\end{align*}
the second equality follows from the fact that 
$H^M(z)= \max\big\{-M, \log |Q(z)|\big\}$ is continuous of logarithmic
growth.
\end{proof}

It will be useful in what follows  to be able to do  diagonal arguments when  dealing with weak convergence and
$\log$-weak convergence.

\begin{lemma} \label{diagonal_argument}
	Consider a measure $\mu\in
  \mathscr{P}(\C )$, and a sequence $(\mu_n)$ and a double
  sequence $(\mu_{n,m})$  in $\mathscr{P}(\C )$. Assume that
	\begin{itemize}
		\item $\mu_n$ converges  weakly to $\mu$  as $n \longrightarrow \infty$,
		\item  for
                  each $n$, $\mu_{n,m} $ converges  weakly to $\mu_{n}$ as $m \longrightarrow \infty$. 
	\end{itemize}
	Then there is a diagonal subsequence $(\mu_{n_i,m_i})_{i \in \N}$ that
        converges weakly to  $\mu$ as $i \longrightarrow \infty$.
	
	Moreover, the same holds if we consider $\mu,\mu_n$ and $\mu_{n,m}$ in
        $\mathscr{P}_{\log}(\C )$ and replace weak convergence
        with $\log$-weak convergence. 
\end{lemma}

\begin{proof}
	Weak convergence on $\mathscr{P}(\C )$ is metrizable (e.g. see \cite[Theorem 6.8]{Billingsley}, \cite[Example
        IV.22-23]{Pollard}), so the first assertion is just the diagonal argument in a  metric space. 

        Let $d_{\weak}$ be a distance in $\mathscr{P}(\C )$ whose
        associated topology is weak convergence. For $\mu ,\mu '\in
        \mathscr{P}_{\log}(\C )$, define
        \begin{displaymath}
          d_{\log^{+}}(\mu ,\mu ')=\left|\int \log^{+}|z|\,\dd \mu - \int \log^{+}|z|\,\dd \mu'\right|
        \end{displaymath}
        and $d=\max\{d_{\weak},d_{\log^{+}}\}$. Then $d$ is a distance
        in $\mathscr{P}_{\log}(\C )$
        and, by  Lemma \ref{estrecha},  its associated topology is
        log-weak convergence. Thus the second statement is again the
        diagonal argument in a  metric space.    
\end{proof}

\subsection{Subharmonic functions and $L^1_{\loc}$ convergence}\label{L1loc}
Let $f, f_n \colon \C  \longrightarrow \R $, $n\ge 0$, be functions that are
locally integrable. We say $f_n$ converges to $f$ \emph{in
  $L^1_{\loc}$}, if for any $z \in \C $, there exists an
open neighborhood $U$ of $x$ such that
\[
\lim_{n \rightarrow \infty} \int_U |f_n-f| \,\dd\lambda = 0,
\]
where $\lambda$ is the usual Lebesgue measure on the complex plane. In
this section we show that the $\log$-weak convergence of measures
corresponds exactly to $L^1_{\loc}$ convergence of
potentials. 

\begin{proposition} \label{log_convergence_IFF_potential_L^1_loc}
	Let $\mu\in\mathscr{P}_{\log}(\C )$ and $(\mu_n)$ be a sequence in $\mathscr{P}_{\log}(\C )$. Then
        $\mu_n$ converges $\log$-weakly to  $\mu$ if and only if 
        $U^{\mu_n}$ converges to $U^\mu$ in
        $L^1_{\loc}$. 
\end{proposition}
\begin{proof}
    We start by the direct implication. By Lemmas \ref{superharmonic potential} and  \ref{Hormander_psh_magic} we
        need only to show that $U^{\mu_n} $ converges to $U^\mu$ as distributions. Taking $\varphi \in C_c^\infty(\C )$, we have that
	\begin{flalign*}
		\lim_{n \rightarrow \infty} \int \varphi(z)
          U^{\mu_n}(z) \,\dd\lambda(z)
          &= \lim_{n \rightarrow \infty} \iint \varphi(z)\log|z-w|^{-1} \dd\mu_n(w) \,\dd\lambda(z) \\
		&=\lim_{n \rightarrow \infty}\iint
                  \varphi(z)\log|z-w|^{-1}\, \dd\lambda(z) \dd \mu_n(w) \\
		&= \iint \varphi(z)\log|z-w|^{-1} \,\dd\lambda(z) \dd\mu(w) \\
		&=\iint \varphi(z)\log|z-w|^{-1}\, \dd\mu(w) \dd\lambda(z) \\
		&=\int \varphi(z) U^\mu(z) \,\dd\lambda(z).
	\end{flalign*}
The second and fourth equalities are Lemma \ref{true_Fubini}. The third equality follows from the fact that the function
        \begin{displaymath}
        w\longmapsto \int \varphi(z)\log|z-w|^{-1} \dd\lambda(z) 
        \end{displaymath}
        is  continuous of  logarithmic growth and the
        hypothesis of $\log$-weak convergence. 
                
   Now we prove the reverse implication. 
   Assume that $U^{\mu_n}$ converges to $ U^\mu$ in 
    $L^1_{\loc}$. It follows that $\Delta U^{\mu_n}$ converges to
    $\Delta U^\mu$ as distributions.  By Lemma \ref{lemm:1}, we 
    conclude that $\mu_n$ converges to $\mu$ as distributions which 
    implies that $\mu_n$ converges to $\mu$ weakly.  By Lemma \ref{estrecha} it
    only remains to be proven that
    \begin{equation}\label{eq:13}
      \lim_{n\to \infty}\int \log^{+}|z|\,\dd \mu _{n}(z)=
      \int \log^{+}|z|\,\dd \mu (z).
    \end{equation}
    Let $\nu =\frac{1}{\pi }\lambda _{B_{1}}$ be the Lebesgue measure
    on the unit ball normalized so that it has total mass one. Let $U^{\nu}$
    be its potential. Then $-U^{\nu}(z)-\log^{+}|z|$ is a continuous and bounded function on $\C$. Since $\mu _{n}$ converges to
    $\mu $ weakly, equation \eqref{eq:13} is equivalent to  
    \begin{equation}\label{eq:25}
      \lim_{n\to \infty}\int U^{\nu}(z)\,\dd \mu _{n}(z)=
      \int U^{\nu}(z)\,\dd \mu (z).
    \end{equation}
Using Lemma
    \ref{Fubini} and the convergence in $L^1_{\loc}$, we
    have that 
\begin{displaymath}
\lim_{n\to \infty} \int U^{\nu}\dd\mu_n = \lim_{n\to \infty}\int
                                          U^{\mu_{n}} \,\dd \nu
                                        = \int U^{\mu} \,\dd \nu
  =\int U^{\nu}\dd\mu, 
\end{displaymath}
obtaining \eqref{eq:25}.
\end{proof}

%Let $\mathscr{P}_{c}^{\overline{\Z }}(\C )$ denote the set of all probability measures that satisfy the equivalent conditions in Theorem \ref{SOS_equidistribution}.

%%% Local Variables:
%%% mode: latex
%%% TeX-master: "main"
%%% End:

\section{The sweetened truncation} \label{sec:log-weak-convergence}

The goal of this section is to define the sweetened truncation of a measure and
prove related technical lemmas and propositions. They
will be used in the next sections. 

\subsection{Sweetened truncation} \label{Smith_truncation_definition}
Let $\mu \in \csP_{\log}(\C )$. For any $R>1$, we want to produce a new measure $\musw_{R}$ supported on the ball
$B_{R}\coloneqq \{|z|\le R\}$ such that $\musw _{R}$ converges to $\mu$ $\log$-weakly as
$R$ goes to $\infty$. A first candidate for this purpose is the \emph{na\"ive truncation}
\begin{displaymath}
  \mu_R^\prime \coloneqq \mu|_{B_{R}} + (1-m_R) \cdot \lambda_{S_R},
\end{displaymath}
where $m_R \coloneqq \mu\left(B_{R} \right)$ and 
$\lambda_{S_R}$ is the equilibrium measure on $S_{R}\coloneqq
\big\{|z|=R\big\}$.

As we will see in Section \ref{sec:sos-equid-non-1}, we are interested
in measures that satisfy
\begin{equation}\label{eq:27}
  \int \log|Q| \,\dd\mu \geq \nobreak 0\quad \text{ for all } Q \in \Z [x]\setminus\{0\}.
\end{equation}
We want to refine the na\"ive truncation in such a way that, if $\mu $
satisfies condition \eqref{eq:27}, then the measure $\musw_{R}$ still
satisfies the same condition. The following definition is inspired by
Smith's concept of \emph{sweetened measure}, see \cite[Definition
5.5]{Smith}.
\begin{definition}\label{def truncation}
Let $\mu \in \csP_{\log}(\C )$ be a probability measure. Let $R>1$ and set
\begin{displaymath}
  T_R \coloneqq  \int_{|z| > R} \log^+|z|
        \dd\mu,\ 
  L_R \coloneqq  (1-m_R) \log2 + T_R,\
  \eta_R \coloneqq \frac{\log R}{\log R+L_R}.
\end{displaymath}
We define the \emph{sweetened truncation} $\musw_R$ as
\begin{equation}\label{truncation formula}
\musw_R \coloneqq \eta_R \cdot \mu|_{B_R} + (1-m_R\eta_R) \cdot \lambda_{S_R}.
\end{equation}
\end{definition}

We note that Jensen's formula implies that the potential of $\lambda
_{S_{R}}$ is given by
\begin{equation}
  \label{eq:28}
  U^{\lambda _{S_{R}}}(z) = -\log \max\big\{ |z|,R \big\}.
\end{equation}

\begin{proposition}\label{truncation properties} Let $\mu \in \csP_{\log}(\C )$. Then
\begin{enumerate}
\item\label{truncation approximates}
	$\musw_R$ converges to $\mu$ $\log$-weakly as $R$ goes to
        $\infty$,
\item\label{truncation conjugation} for all $R>1$, if $\mu $ is
  conjugation invariant, then $\musw_R$ is conjugation invariant as
  well,
\item  \label{Smith_truncation_keeps_Smith_conditions}
for all $R>1$, the inequality $U^{\musw_R} \leq \eta_R U^\mu$ holds. In
particular, if $ \int \log|Q| \,\dd\mu \geq \nobreak 0$ for some $Q
\in \Z [x]\setminus\{0\}$, then $ \int \log|Q| \,\dd\musw_R \geq
\nobreak 0$ as well. 
\end{enumerate}
\end{proposition}

We will prove this statement after the following lemma.

\begin{lemma} \label{hangover_estimate}
	The na\"ive truncation  $\mu_R^\prime$ satisfies $U^{\mu_{R}^\prime} \leq U^\mu + L_R$.
\end{lemma}

\begin{proof}
	Assume $R>1$. By equation \eqref{eq:28}, 
	\[
	U^{\mu_R^\prime}(z)-U^{\mu}(z) = \int_{|w|>R} \log |z-w| \,\dd\mu(w)-(1-m_R)\log \max\big\{|z|,R\big\}.
	\]

Assume $|z|<R$. Then for all $w$ with $|w|>R$, we have $|z-w| \leq |z|+|w| \leq 2|w|$ and hence
\begin{align*}
U^{\mu^\prime_R}(z)-U^{\mu}(z)&=\int_{|w|> R}\log |z-w| \,\dd \mu(w)-(1-m_R)\log R\\
&\leq \int_{|w|> R}\log |2w| \,\dd \mu(w) \\
&\leq L_R.
\end{align*}

Assume now $|z| \geq R$. Then,
\begin{align*}
U^{\mu_R^\prime}(z)-U^{\mu}(z)&=\int_{|w|> R}\log |z-w| \,\dd \mu(w)-(1-m_R)\log |z|\\
&= \int_{|w|> R}\log \left|1-\frac{w}{z}\right| \,\dd\mu(w)\\
&=\int_{|w|>|z| \geq R}\log \left|1-\frac{w}{z}\right| \,\dd \mu(w) + \int_{|z|\geq |w|>R}\log \left|1-\frac{w}{z}\right|\,\dd \mu(w)\\
&\leq \int_{|w|>|z|\geq R} \log \left|\frac{2w}{R}\right| \,\dd\mu(w)
  + \int_{|z|\geq |w|>R} \log 2 \,\dd\mu(w)\\
&=  \int_{|w|>|z|\geq R} \log \left|\frac{w}{R}\right| \,\dd\mu(w)
  + \int_{|w|>R} \log 2 \,\dd\mu(w)\\
&\leq \int_{|w|> R} \log |w| \,\dd \mu(w) + \int_{|w|>R} \log 2 \,\dd \mu(w)\\
&\le L_R.
\end{align*}
\end{proof}

\begin{proof}[Proof of Proposition \ref{truncation properties}] Item
  \eqref{truncation conjugation} follows readily from the
  definitions. Now we  show \eqref{truncation approximates}. It is
  easy to see from \eqref{truncation formula}  that $\musw_R$
  converges to $\mu$ weakly as $R$ goes to $\infty$. So by Lemma
  \ref{estrecha}, it suffices to test against $\log^+|z|$: 
	\begin{align*}
		\int &\log^+|z| \,\dd\musw_R - \int \log^+|z| \,\dd\mu \\
		&= (1-m_R\eta_R) \log R + (\eta_R-1) \int_{|z| \leq R}
                  \log^+|z| \,\dd\mu - \int_{|z| > R} \log^+|z|
                  \,\dd\mu.
	\end{align*}
	Clearly the last two terms vanish when
        $R$ goes to $\infty$. With regards to the first term, we
        have that  
\begin{equation}\label{otro}
(1-m_R\eta_R) \log R = \big( (1-m_R)\log R +L_R\big) \frac{\log R}{\log R+L_R}.
\end{equation}
Then, the estimate 
$0\leq (1-m_R)\log R\leq T_R$ and the facts that $T_R$ and $L_R$
are positive and converge to zero when $R$ goes to $\infty$ imply
that $(1-m_R\eta_R) \log R$ also converges to zero. 

Finally we prove \eqref{Smith_truncation_keeps_Smith_conditions}.	For the first assertion, we use Lemma \ref{hangover_estimate}
\begin{flalign*}
	U^{\musw_R} &= \eta_R U^{\mu_R^\prime} +(1-\eta_R) U^{\lambda_{S_R}} \\
	&\leq \eta_R U^{\mu} + \eta_R L_R - (1-\eta_R) \log R  \\
	&\leq \eta_R U^{\mu}.
\end{flalign*}

For the second one, let $a \in \Z $ be the leading coefficient of $Q$
and $\alpha _1,\dots,\alpha _n$ be the roots of $Q$. Then 
\begin{flalign*}
	\int \log|Q| \,\dd\musw_R &= \log |a| - \sum U^{\musw_R} (\alpha _i) \\
	& \geq \eta_R \log|a| - \eta_R \sum U^{\mu} (\alpha _i) \\
	& \geq \eta_R \int \log|Q| \,\dd\mu \\
	& \geq 0.
\end{flalign*}
\end{proof}

\subsection{Sweetened truncation and asymptotically logarithmic functions}
\label{sec:trunc-funct-with}
\begin{definition}\label{log singular}
A function $\ginf \colon \C  \longrightarrow \R $ is called
\emph{asymptotically logarithmic at infinity}, if it obeys the asymptotic
	\[
	\ginf(z) = \log|z| + o(\log|z|), \quad \text{as } |z| \longrightarrow \infty.
	\]
A \emph{Green function} is a function $\ginf \colon \C  \longrightarrow \R $ which is continuous, invariant under complex conjugation and asymptotically logarithmic at infinity. Such Green functions arise when considering singular metrics on
$\mathcal{O}(1)$, see Section \ref{Arakelov_theory_on_P^1} for a discussion.
\end{definition}

The purpose of this section is to prove the following statement that
will play an essential role in the proof of strong duality in Section \ref{section_strong_duality}.

\begin{proposition}\label{truncation-log infinity}
Let $\ginf \colon \C \longrightarrow \R$ be a continuous function that
is asymptotically logarithmic at infinity. Let $(\mu_n)$ be a sequence in $\mathscr{P}(\mathbb{C})$ such that 
\begin{displaymath}
  \limsup_{n\to \infty} \int \ginf \,\dd\mu_n<\infty.
\end{displaymath}
Then, for all $\varepsilon>0$ there exists $R_0>1$,
such that for all $R\ge R_{0}$
\begin{equation}\label{control}
\limsup_{n\to \infty} \int \ginf \,\dd\musw_{n,R} \leq \limsup_{n\to
  \infty} \int \ginf \,\dd\mu_{n} +\varepsilon.
\end{equation} 
\end{proposition}

%----------------------------------------

\begin{proof}
The hypotheses imply 
\begin{displaymath}
	\limsup_{n \rightarrow \infty} \int \log^+|z| \,\dd\mu_n <\infty,
\end{displaymath}
so we may assume that
\begin{equation}{\label{log tight}}
	\sup_{n\in \N} \int \log^+|z| \,\dd\mu_n <\infty.
\end{equation}

Since $1=o(\log|z|)$ when $|z| \longrightarrow \infty$, we have
$( \mu_n )$ is tight by Lemma \ref{tight_transition}. Using
Prohorov's theorem \cite[Theorem 5.1]{Billingsley},   we can assume,
after taking a subsequence if necessary, that there is a probability measure
$\mu_\infty$ on $\C$ such that $\mu_n $ converges weakly to
$\mu_\infty$. In particular, we know that $\mu_\infty$ can integrate 
$\log^+|z|$, because
\begin{equation}\label{log integrable}
\int \log^+|z| \,\dd\mu_\infty \leq \limsup_{n\to \infty} \int
\log^+|z| \,\dd\mu_n < \infty.
\end{equation}

Set $E_{n,R}\coloneqq\int \ginf \,\dd\musw_{n,R} - \int \ginf
\,\dd\mu_n$. In order to prove \eqref{control} it is enough to show
that  
\begin{equation}\label{goal1}
\limsup_{R \rightarrow \infty} \limsup_{n \rightarrow \infty} E_{n,R} \leq 0
\end{equation}

Before the proof of \eqref{goal1} we show several auxiliary results.
First we have that 
\begin{equation}
  \label{eq:21}
  \lim_{R \rightarrow \infty}\limsup_{n\rightarrow \infty} \mu_n\{|z|\geq R\} \log R=0.
\end{equation}
Indeed,  
\begin{align*}
	0&\leq \limsup_{n \rightarrow \infty} \mu_n \big\{ |z| \geq R\big\} \log R \\
	&\leq  \mu_\infty\big\{|z| \geq R\big\} \log R \\
	&\leq  \int_{|z|\geq R} \log^+|z| \,\dd\mu_{\infty}.
\end{align*}
Here the second inequality follows from the Portmanteau theorem
\cite[Theorem 2.1 (iii)]{Billingsley}, because $\mu_n$ converges
weakly to $ \mu_\infty$ and the set $\big\{ |z| \geq R \big\}$ is
closed. Since \eqref{log integrable} implies that the last term
vanishes when $R\longrightarrow \infty$, this proves the claim.

Let $L_{n,R}$ and $T_{n,R}$ be the constants appearing in Definition
\ref{def truncation} for the sweetened truncation of $\mu _{n}$, $n\ge 1$. 
There exists $M>0$ such that 
\begin{equation}\label{eq:22}
  \big\{ L_{n,R}, T_{n,R} \mid n\geq 1, R>1 \big\}\subseteq [0,M].
\end{equation}
We see directly from the definitions  that  $L_{n,R}, T_{n,R}\geq
0$. Also, $T_{n,R}$ is bounded because of \eqref{log tight}. Since
  $L_{n,R}\le T_{n,R}+\log 2$ it is bounded as well.
% On the
% other hand,  {\color{red} J: since $L_R\le 2\log 2 +T_R$ the next
%   computetion seems superfluous}
% \begin{align*}
%   L_{n,R}&=2(1-m_{n,R})\log 2 + \int_{|z| \geq R} \log^+|z| \,\dd\mu_n \\
%          &\leq 2\log 2 + \int \log^+|z| \,\dd\mu_n \\
%          &\leq 2\log2 + \sup_k \int \log^+|z| \,\dd\mu_k < \infty,
% \end{align*}
% by \eqref{log tight}. 

We have that
\begin{equation}\label{eq:23}
  \limsup_{R \rightarrow \infty} \limsup_{n \rightarrow \infty} \left( (1-m_{n,R}\eta_{n,R}) \log R - T_{n,R}\right)\leq0.
\end{equation}

Indeed,  since $\log R$ and $ L_{n,R}$ are non negative,  we see that 
\begin{align*}
(1-m_{n,R}\eta_{n,R}) \log R - T_{n,R}  &= \frac{\log R}{\log R+L_{n,R}}\cdot\left((1-m_{n,R})\log (2R) +T_{n,R} \right) -T_{n,R}\\
&\leq \big((1-m_{n,R})\log (2R) +T_{n,R} \big) -T_{n,R}\\
                                        &\le (1-m_{n,R})\log (2R)\\
  &\le \mu_n \big\{ |z| \geq R\big\} \log (2R)
\end{align*}
We conclude by \eqref{eq:21}.

We next observe that there exists $M'>0$ such that for all $n\geq 1$
and $R>0$ 
\begin{equation}
  \label{eq:24}
  (1-m_{n,R}\eta_{n,R}) \log R \leq M'
\end{equation}
Holds. This follows from \eqref{eq:23} and \eqref{eq:22}.

We have the estimate 
\begin{equation}
  \label{eq:26}
  0\leq 1-\eta_{n,R}\leq \frac{M}{\log R}.
\end{equation}
That follows from $1-\eta_{n,R}=\frac{L_{n,R}}{\log R + L_{n,R}}$ and \eqref{eq:22}.

Write $\ginf = \log^+|z| + \psi $, where $\psi: \mathbb{C} \longrightarrow \mathbb{R}$ is continuous and $\psi=o(\log|z|)$ at infinity. Note
that the sequence $(\int \psi  \,\dd\mu_n)$ is bounded because of \eqref{log tight}.

We now put together all the previous results.
 Let $\varepsilon>0$. Since $\psi =o(\log |z|)$, there exists $R_1>1$ such
 that for all $z$ with $|z|\geq R_0$, we have that $\psi (z)\leq
 \frac{\varepsilon}{M'} \log |z|$. For $R\geq R_1$ and using the
 estimate \eqref{eq:24}, we have  
\begin{align*}
E_{n,R} =&(1-m_{n,R}\eta_{n,R})\int \ginf \,\dd\lambda_{S_R}  -
           \int_{|z|> R}\ginf \,\dd\mu_n - (1-\eta_{n,R})\int_{|z|\leq
           R}\ginf \,\dd\mu_n  \\ 
  =& \big((1-m_{n,R}\eta_{n,R}) \log R  - T_{n,R} \big)- (1-\eta_{n,R}) \int_{|z| \leq R} \log^+|z| \,\dd\mu_n  \\
	& \quad \quad   + (1-m_{n,R}\eta_{n,R}) \int \psi  \,\dd\lambda_{S_R}- \int_{|z|> R} \psi  \,\dd\mu_n  - (1-\eta_{n,R}) \int_{|z| \leq R} \psi  \,\dd\mu_n.\\
	\leq & \big((1-m_{n,R}\eta_{n,R}) \log R  - T_{n,R} \big) + \varepsilon    - \int_{|z|> R} \psi  \,\dd\mu_n  - (1-\eta_{n,R}) \int_{|z| \leq R} \psi  \,\dd\mu_n.\\
	=&  \big((1-m_{n,R}\eta_{n,R}) \log R  - T_{n,R} \big) +
           \varepsilon -\eta_{n,R} \int_{|z|> R} \psi \,\dd\mu_n  - (1-\eta_{n,R}) \int \psi  \,\dd\mu_n .
\end{align*}
By Lemma \ref{tight_transition} we know that $(\mu_n)$ is
$\psi $-tight. Combining this information with the boundedness of $(\int \psi 
\,\dd\mu_n)$ and the estimates \eqref{eq:23} and \eqref{eq:26}, we obtain
\begin{displaymath}
  \limsup_{R\rightarrow \infty} \limsup_{n\rightarrow \infty} E_{n,R}\leq \varepsilon.
\end{displaymath}
Since $\varepsilon>0$ is arbitrary, this shows \eqref{goal1}.
\end{proof}

%%% Local Variables:
%%% mode: latex
%%% TeX-master: "main"
%%% End:

\section{Approximating measures by algebraic integers and by capacity one compact sets} \label{sec:sos-equid-non-1}

\subsection{Approximating measures by algebraic integers in the non-compact setting}

The aim of this section is to prove the following  result.

\begin{theorem} \label{SOS non compact} Let
  $\mu \in \mathscr{P}_{\log}(\C )$. Then
  $\mu \in \mathscr{P}^{\overline{\Z }}_{\log}(\C )$ if and only if it
  is invariant under the complex conjugation and
        \begin{equation}\label{S condition}
        \int \log|Q| \,\dd\mu \geq
        0 \quad  \textrm{ for all } Q \in \Z [x]
        \setminus \{0\}.
                  \end{equation}
\end{theorem}

Combining this result  with Theorem \ref{thm:2} we readily obtain the following consequence for measures with compact support.

\begin{corollary}
Let $\mu\in \mathscr{P}_c(\C )$. Then, the following are equivalent:
	\begin{enumerate}
		\item there exists a sequence of
                  distinct algebraic integers $(\alpha _n ) $ such that $\delta_{O(\alpha _n)}$ converges properly
                  to  $ \mu$;
		\item there exists a sequence of
                  distinct algebraic integers $(\alpha _n ) $ such that $\delta_{O(\alpha _n)}$ $\log$-weakly converges 
                  to  $ \mu$.
		
	\end{enumerate}
\end{corollary}

 Hence Theorem \ref{SOS non compact}
  can be considered as an extension of the Smith--Orloski--Sardari
  theorem to non-compactly supported measures that can integrate
  functions of logarithmic growth. This further justifies our claim that the
  notion of $\log$-weak convergence is the right one for non-compactly
  supported measures.

  On the other hand, weak convergence is not appropriate in this
  setting because any probability measure that is invariant under the
  complex conjugation is the weak limit of a sequence of Galois orbits
  of distinct algebraic integers.  In particular such weak limits do not necessarily satisfy
  condition \eqref{S condition}, see Section~\ref{Nuno} for an example
  in this direction.

\begin{remark}
  If $(\alpha _n )$ is a sequence of distinct algebraic integers such
  that $O(\alpha _n)$ converges $\log$-weakly to $\mu$ then
  $\deg(\alpha _n)$ converges to $\infty$, as it can be seen using the
  argument in Remark~\ref{degree_unbounded}.
\end{remark}

\begin{proof}[Proof of Theorem \ref{SOS non compact}]
  First assume that $\mu \in \csP^{\overline{\Z}}_{\log}(\C)$. Then
  this measure is clearly invariant under the complex conjugation, and
  we next prove that it satisfies condition~\eqref{S condition} by
  adapting to our setting a classical argument appearing for instance in 
  \cite[Lemma~1.3.4]{Serre19}.  Let $(\alpha _n)$ be a sequence of
  distinct algebraic integers such that $O(\alpha _{n})$ converges
  $\log$-weakly to $\mu $ and denote by $P_n(x) \in \Z[x]$ the minimal
  polynomial of $\alpha _n$. Let $Q\in \Z[x]$ be a nonzero
  polynomial. By Lemma \ref{monotone} we have that
\begin{displaymath}
  \int \log |Q|\dd \mu \geq \liminf_n  \frac{1}{\deg (\alpha
    _n)}\sum_{\beta \in O(\alpha _n)} \log |Q(\beta )|.
\end{displaymath}
Let $\Res(P_n,Q)$ be the resultant of $P_n$ and $Q$.  We have that
$\Res(P_n,Q) =\prod_{\beta \in O(\alpha _n)}Q(\beta
  )$ because $P_n$ is monic
\cite[Section B.1.13]{BombieriGubler}. Hence
\begin{displaymath}
  \sum_{\beta \in O(\alpha _n)}  \log
  |Q(\beta )|=\log |\Res(P_n,Q)|\geq 0 \quad \text{ for all } n\gg 0,
\end{displaymath}
the last inequality being due to the fact that
$\operatorname{Res}(P_n,Q)$ is a nonzero integer when $n$ is
big~enough.

Conversely, let $\mu \in \mathscr{P}_{\log}(\C )$ invariant under the
complex conjugation and satisfying condition \eqref{S condition}. For
$R>1$ consider the sweetened truncation $\musw_R$ (Definition \ref{def
  truncation}). By Proposition~\ref{truncation
  properties}(\ref{Smith_truncation_keeps_Smith_conditions},\ref{truncation
  conjugation}) we can apply Theorem \ref{thm:2} to $\musw_R$ to
conclude that it can be approximated properly (and in particular
$\log$-weakly) by a sequence of Galois orbits of distinct algebraic
integers. Then using Proposition~\ref{truncation
  properties}\eqref{truncation approximates} and Lemma
\ref{diagonal_argument} we can find a diagonal subsequence that
converges to $\mu$ $\log$-weakly, as desired.
\end{proof}

%\todo[inline]{Suppressed \emph{``The classical Fekete--Szeg\H{o} theorem (Theorem \ref{Fekete-Szego}), plus
%Rumely's remark on Bilu's equidistribution (Remark
%\ref{Rumely_equidistribution}), imply that if $K \subseteq \C $
%is a compact subset that is invariant under complex conjugation and
%has capacity one, then $\mu_K \in
%\mathscr{P}_{\log}^{\overline{\Z }}(\C )$. The
%Smith--Orloski--Sardari theorem (Theorem
%\ref{thm:2}) characterizes all compactly supported ones''}}

% Before proving Theorem \ref{SOS non compact} we state a   direct consequence.

% \begin{corollary}\label{negative potential}\todo{Add this reference to Theorem C?}
%   Let $\mu \in \mathscr{P}_{\log}(\C )$ be a conjugation invariant
%   measure such that $U^\mu\leq 0$. Then
%   $\mu \in \mathscr{P}_{\log}^{\overline{\Z }}(\C )$.
% \end{corollary}

% \begin{proof}
%   Let $Q =a \prod_{i=1}^{n} (X-\alpha_{i})\in \mathbb{Z}[X] \setminus\{0\}$. Then
%   \begin{displaymath}
%     \int \log|Q| \,\dd\mu = \log |a| - \sum_{i=1}^{n} U^{\mu} (\alpha _i) \ge 0
%   \end{displaymath}
%   because $a\in \mathbb{Z} \setminus \{0\}$ and $U^{\mu}\le 0$.  Hence
%   $\mu$ satisfies the condition \eqref{S condition} and we conclude by
%   Theorem \ref{SOS non compact}.
% \end{proof}

This result shows that $\mathscr{P}^{\overline{\Z}}_{\log}(\C )$ is a
convex subset of $\mathscr{P}_{\log}(\C )$.  We next construct an
explicit family of measures that form a countable dense subset
it. To this end, for each pair
$P,Q$ of nonconstant  coprime polynomials with integer
coefficients we consider the~map
\begin{displaymath}
  \varphi_{P,Q}\colon \PP^1(\mathbb{C}) \longrightarrow \mathbb{P}^1(\mathbb{C})
\end{displaymath}
induced by the rational function $ P^{\deg(Q)+1}/Q^{\deg(P)}$.  It is
finite of degree ${\deg(P) (\deg(Q)+1)}$.

Given a function $f \colon \PP^1(\mathbb{C}) \longrightarrow \R$ we
consider its \emph{pushforward}
$(\varphi_{P,Q})_*f \colon \PP^1(\mathbb{C}) \longrightarrow \R$
defined as
\begin{displaymath}
(\varphi_{P,Q})_*f(z)=\sum_{w\in \varphi_{P,Q}^{-1}(z)} e_w f(w)
\end{displaymath}
with $e_w$ the ramification index of $\varphi_{P,Q}$ at a point $w$.
If $f$ is continuous then this is also the case for its pushforward,
by the continuity of the roots of polynomials with respect to
variations of their coefficients. Hence
for a measure~$\mu$ on $\mathbb{P}^{1}(\C)$ we can define its
\emph{pullback measure} $\varphi_{P,Q}^*\mu$ on $\mathbb{P}^{1}(\C)$
by the rule
\begin{displaymath}
  \int f \, \dd(\varphi_{P,Q}^*\mu) := \int (\varphi_{P,Q})_* f \, \dd \mu \quad \text{ for every continuous }
  f \colon \PP^1(\mathbb{C}) \longrightarrow \R.
\end{displaymath}
The supports of these measures are related by
$\supp(\varphi_{P,Q}^{*} \mu)= \varphi_{P,Q}^{-1}(\supp(\mu))$.  We have that
$\varphi_{P,Q}(\infty)=\infty$, and so if the measure $\mu$ is supported on
$\mathbb{C}\simeq \mathbb{P}^{1}(\mathbb{C})\setminus \{\infty\}$ then
this is also the case for its pullback.

Finally we consider the
measure on $\mathbb{C}$ defined as
\begin{equation}
  \label{eq:31}
  \mu_{P,Q} = \frac{(\varphi_{P,Q})^* \lambda_{S_{1}}}{\deg(P) (\deg(Q)+1)}
\end{equation}
with $\lambda_{S_{1}}$ the equilibrium measure of the unit circle.

\begin{proposition}
  \label{prop:6}
We have that  $\mu_{P,Q}$ is a probability measure supported 
 on the compact subset
\begin{equation}\label{compact set}
\left\{ z \,\middle|\,  |P(z)|^{\deg(Q)+1}=|Q(z)|^{\deg(P)} \right\} 
\end{equation}
with potential function $\displaystyle{U^{\mu_{P,Q}} = -\max \left\{ \frac{\log|P|}{\deg(P)}, \frac{\log|Q|}{\deg(Q)+1} \right\}}$.
\end{proposition}

\begin{proof}
  The support of $\mu_{P,Q}$ is the preimage of $S_{1}$ with respect
  to the map $\varphi_{P,Q}$, which coincides with \eqref{compact
    set}. Next set for short $d=\deg(P)$ and $e=\deg(Q)$. Then for each
  $z\in \C$~we~have
  \begin{equation}
    \label{eq:14}
    U^{\mu_{P,Q}}(z)=\int \log |z-w|^{-1}  \dd \mu_{P,Q}(w)
    = \frac{-1}{d\, (e+1)} \int \sum_{\varphi_{P,Q}(w)=y} e_w\log|z-w| \, \dd \lambda_{S_{1}}(y).
  \end{equation}
  We have
  \begin{displaymath}
    \sum_{w\in \varphi_{P,Q}^{-1}(z)} e_w\log|z-w| = \log \Bigg|  \prod_{w\in \varphi_{P,Q}^{-1}(z)} (z-w)^{e_w}\Bigg|= \log |P(z)^{e+1}-y\, Q(z)^{d}|
  \end{displaymath}
  and so the formula for the potential follows from \eqref{eq:14}
  together with Jensen's formula.
\end{proof}

  \begin{theorem} \label{mu_PQ_construction} For all  $P,Q \in \mathbb{Z}[X]$, with $P$ non constant, and $P,Q$ having no common roots, we have
    that $\mu_{P,Q} \in \mathscr{P}^{\overline{\Z}}_{\log}(\C
    )$. Moreover the family
  \begin{equation}
    \label{eq:29}
    \{ \mu_{P,Q} \, |\ P,Q \in \mathbb{Z}[X]\setminus \Z \text{ with } P\ne Q \text{ monic and irreducible}\}
  \end{equation}
  is a countable dense subset of
  $\mathscr{P}^{\overline{\Z}}_{\log}(\C )$.
\end{theorem}

\begin{proof}
  Let $P,Q \in \mathbb{Z}[X]\setminus \{0\}$ coprime and take an
  primitive irreducible polynomial $F \in \Z[x]$ with leading
  coefficient $a$ and roots $\alpha_{1},\dots, \alpha_{n}$. Using
  Proposition~\ref{prop:6} we get
  \begin{flalign*}
		\int \log|F| \, \dd\mu_{P,Q} &= \log|a| - \sum_i U^{\mu_{P,Q}}(\alpha_i) \\
		&=\log|a|+ \sum_i \max \left\{ \frac{\log|P(\alpha
                  _i)|}{\deg(P)}, \frac{\log|Q(\alpha _i)|}{\deg(Q)+1}
                  \right\} \\ 
		&\geq \max \left\{\log|a|+  \sum_i \frac{\log|P(\alpha
                  _i)|}{\deg(P)}, \log|a|+ \sum_i
                  \frac{\log|Q(\alpha _i)|}{\deg(Q)+1} \right\} \\ 
		&\geq \max \left\{
                  \frac{\log|\operatorname{Res}(F,P)|}{\deg(P)},
                  \frac{\log|\operatorname{Res}(F,Q)|}{\deg(Q)+1}
                  \right\} \\ 
		&\geq 0,
  \end{flalign*}
  because $\operatorname{Res}(F,P)$ and $\operatorname{Res}(F,Q)$ are
  integers that cannot be both zero as $P$,$ Q$ are coprime and $F$ is
  irreducible.  Since $\mu_{P,Q} $ is invariant under complex
  conjugation, Theorem~\ref{SOS non compact}  implies that
  $\mu_{P,Q} \in \mathscr{P}^{\overline{\Z}}_{\log}(\C )$.
	
  Now let $\mu \in \mathscr{P}^{\overline{\Z}}_{\log}(\C )$ and choose
  a sequence $(\alpha _n)$ of distinct algebraic integers such that
  $O(\alpha _n)$ converges to $\mu$ $\log$-weakly. Let
  $P_n \in \mathbb{Z}[X]$ be the minimal polynomial of $\alpha
  _n$. By Proposition \ref{log_convergence_IFF_potential_L^1_loc}
  we have that
	\[
	-\frac{\log|P_n|}{\deg(P_n)} \text{ converges to } U^\mu \ 
        \text{in $L^1_{\loc}$},
      \]
which implies that 
	\[
	-\max \left\{ \frac{\log|P_n|}{\deg(P_n)},
          \frac{\log|P_{n+1}|}{\deg(P_{n+1})+1} \right\}
        \text{ converges to } U^{\mu} \  \text{in $L^1_{\loc}$}.
      \]
      By Proposition \ref{log_convergence_IFF_potential_L^1_loc} and
      \ref{prop:6} we have that $\mu_{P_n,P_{n+1}} $ converges to
      $\mu$ $\log$-weakly, proving the density of the family
      \eqref{eq:29} with respect to the $\log$-weak topology.
\end{proof}

\begin{remark}\label{rem:3}
  One could alternatively consider for each pair $P,Q$ in
  \eqref{eq:29} the measure $\mu^\prime_{P,Q}$ similarly induced by
  the rational function $ P^{\deg(Q)}/Q^{\deg(P)}$. This gives another
  countable dense subset of $\mathscr{P}^{\overline{\Z}}_{\log}(\C )$
  that might look more natural. The advantage of the measures
  $\mu_{P,Q}$ is that they are compactly supported, which is important
  for its later use in the proof of Proposition \ref{prop:7}.
\end{remark}

% \begin{remark} We can also define $\mu _{P,Q}$ for any pair $P,Q\in
%   \Z[x]$ coprime and the argument in the first half of the proof of Proposition
%   \ref{mu_PQ_construction} shows that $\mu_{P,Q}\in
%   \mathscr{P}^{\overline{\Z}}_{\log}(\C )$. The same is true for the
%   alternative measures discussed in Remark \ref{rem:3}. This gives
%   many examples of measures in $\mathscr{P}^{\overline{\Z}}_{\log}(\C
%   )$.  
% \end{remark}

%----------------------------------------
\subsection{Approximating measures by equilibrium measures of compact sets of capacity one}
\label{sec:negative-potentials}

The next result characterizes the measures in
$ \mathscr{P}_{\log}(\C )$ with negative potential as the $\log$-weak
closure of the set of equilibrium measures of compact subsets of
capacity one.

\begin{theorem} \label{negativo} Let
  $\mu \in \mathscr{P}_{\log}(\C )$. Then the following are
  equivalent:
	\begin{enumerate}
        \item \label{item:5} there exists a sequence of compact
          subsets $K_n \subseteq \mathbb{C}$ with
          $\mathrm{cap}(K_{n})=1$ such that $\mu_{K_n}$ converges
          $\log$-weakly to $\mu$;
		\item \label{item:6} $U^{\mu} \leq 0$.
        \end{enumerate}
        In particular, if $\mu$ is invariant under the complex
        conjugation and $U^{\mu}\le 0$ then
        $\mu\in
        \mathscr{P}_{\log}^{\overline{\mathbb{Z}}}(\mathbb{C})$.
\end{theorem}

\begin{proof}
We show first that \eqref{item:5} implies  \eqref{item:6}. Using the
monotone convergence theorem we get
\begin{align*}
U^\mu(z) &= \int \log|z-w|^{-1}  \dd\mu(w)\\
&= \lim_{M \rightarrow \infty} \int \min \left\{ M, \log|z-w|^{-1}\right\} \dd \mu(w)\\
         &= \lim_{M \rightarrow \infty} \lim_{n \rightarrow \infty} \int  \min \left\{ M, \log|z-w|^{-1}\right\} \dd \mu_{K_n}(w) \\[1mm]
  & \le \lim_{n\to \infty} U^{\mu_{K_n}}(z),
\end{align*}
where the third step follows from the fact that the function
$ w \mapsto \min \left\{ M, \log |z-w|^{-1}\right\}$ is continuous and of
logarithmic growth.  By Frostman's theorem \cite[Theorem
3.3.4(a)]{Ransford} 
\begin{displaymath}
U^{\mu_{K_n}}(z) \le -\log(\mathrm{cap}(K_{n}))=0  \quad \text{ for all } n
\end{displaymath}
and so $U^\mu(z)\leq 0$ as stated.

Now we prove that \eqref{item:6} implies \eqref{item:5}. Using
  Proposition \ref{truncation properties}\eqref{truncation
    approximates} and Lemma \ref{diagonal_argument} we  restrict
  without loss of generality to the case when $\mu$ is supported in a
  compact subset $K$. Then we can weakly approximate
  $\mu$ by a sequence of discrete  measures 
\begin{displaymath}
  \mu_n = \frac{1}{b_{n}} \sum_{s\in S_n}
  \delta_s,
\end{displaymath}
where $S_{n}$ is a finite subset of $K$ of cardinality
$b_{n} \in \mathbb{Z}_{>0}$.  Then $(\mu_n)$ converges properly to
$\mu$, and so it also converges log-weakly. By Proposition
\ref{log_convergence_IFF_potential_L^1_loc} we have that $U^{\mu_{n}}$
converges to $ U^\mu$ in $L^1_{\loc}$, and so
% \begin{equation}
%   \label{eq:30}
%   \min\{U^{\mu_{n}},0\} \text{ converges to } \min\{U^\mu,0\}=U^{\mu} \quad \text{ in }
%  L^1_{\loc}.
% \end{equation}
$\min\{U^{\mu_{n}},0\}$ converges to $ \min\{U^\mu,0\}=U^{\mu}$ in
$ L^1_{\loc}$.

Setting $P_n(X)=\prod_{\alpha \in S_n} (X-\alpha)$ we have that 
$U^{\mu_{n}}=-{\log|P_n|}/{\deg(P_n)}$ and so 
\begin{displaymath}
  \min\{U^{\mu_{n}},0\} = \min\left\{-\frac{\log|P_n|}{\deg(P_n)},0\right\} =U^{\mu_{K_{n}}}
\end{displaymath}
for the  compact subset
$K_n \coloneqq \big\{ z \mid |P_n(z)| \leq 1 \big\}$, which is of
capacity one because $P_n$ is monic.  Applying again Proposition
\ref{log_convergence_IFF_potential_L^1_loc} we conclude that
$\mu_{K_{n}}$ converges log-weakly to $\mu$, as stated. 

The last statement follows readily from the Fekete-Szeg\H{o} theorem
combined with Rumely's equidistribution result (Remark
\ref{Rumely_equidistribution}) together with Lemma \ref{diagonal_argument}.
\end{proof}

Using this result it is easy to find measures in
$\mathscr{P}_{\log}^{\overline{\Z}}(\C)$ that are not compactly
supported. For instance the Fubini-Study form
\[
\omega_{\FS} =\frac{i \, \dd z \wedge \dd \bar{z}}{2\pi \, (|z|^2+1)^2}
\]
induces a conjugation invariant measure $\mu_{\FS}\in
\csP_{\log}(\C )$ with
potential
\begin{displaymath}
  U^{\mu_{\FS}}(z)=-\frac{1}{2}\log(|z|^2+1)\le 0,
\end{displaymath}
and so Theorem \ref{negativo} gives that
$\mu_{\FS}\in \csP_{\log}^{\overline{\Z }}(\C )$.

Furthermore, combining this result with Theorem
\ref{mu_PQ_construction} it is also easy to produce measures in
$\csP_{\log}^{\overline{\Z }}(\C )$ that cannot be approached by
equilibrium measures of compact subsets of capacity one, e.g. see example \eqref{exampleintro}.

%%% Local Variables:
%%% mode: latex
%%% TeX-master: "main"
%%% End:

\section{Strong duality} \label{section_strong_duality}

Here we present two optimization problems that will be crucial for our
study of heights of algebraic points.  

%\todo[inline]{Deleted ``They are akin to those considered by Smyth
 % \cite{Smyth:tpaist} and Smith \cite{Smith} in their works on totally
  %real algebraic integers of small trace, which were indeed a major
  %source of inspiration for the material in this section'' (already
  %said in the intro)}

Let $\ginf\colon \C  \longrightarrow \R $ be a continuous
function that is asymptotically logarithmic at $\infty$ in the sense of
Definition~\ref{log singular}, and for simplicity fix an enumeration
$Q_{1}, Q_{2}, Q_{3}, \dots$ of all nonconstant primitive irreducible
polynomials with integer coefficients. 
We  consider the \emph{primal problem}   and  the \emph{dual problem}
respectively defined~as 
\begin{equation}
  \label{eq:51}
   \primal(\ginf) = \inf \left\{ \int \ginf\, \dd\mu \ \,\middle|\, \ \mu \in \mathscr{P}_{\log}(\C ), \, \int \log|Q_{n}| \, \dd \mu \ge 0 \text{ for all } n\in \mathbb{N}\right\} 
  \end{equation}
and
 \begin{equation}
    \label{eq:6}
    \dual(\ginf) =
    \sup_{(a_{n})}   \inf_{z\in \C } \left( \ginf(z) - \sum_{n} a_{n}\log |Q_{n}(z)| \right) ,
  \end{equation}
  the supremum being over the sequences $(a_{n})$
  in $ \R _{\ge 0}$ with $a_{n}=0$ for all
  but a finite number of $n$'s.  We refer to the quantities
  $\mathcal{P}(\ginf)$ and $\dual(\ginf)$ as the
  \emph{optimal values} for~these~problems.

\begin{remark}
  \label{rem:21}
If $\ginf$ is invariant under the complex conjugation then
\begin{displaymath}
  \primal(\ginf)= \inf \Big\{ \int \ginf \,\dd\mu \, \Big| \
  \mu \in \mathscr{P}^{\overline{\Z }}_{\log}(\C )
  \Big\}.
\end{displaymath}
Indeed, in this case the primal problem can be computed over the
measures  $\mu\in \mathscr{P}_{\log}(\C )$ that  are invariant under
the complex conjugation and satisfy
$\int \log|Q_{n}| \, \dd \mu \ge 0 $ for all $n$, which by
Theorem~\ref{SOS non compact} coincide with those in
$ \mathscr{P}^{\overline{\Z }}_{\log}(\C )$.
\end{remark}

As explained in Appendix \ref{sec:dual-line-optim}, this pair of
problems is an instance of primal and dual problems in linear
optimization (Example \ref{exm:1}).  In particular, they satisfy the
weak duality property.

\begin{proposition}
  \label{prop:1} We have
  $\mathcal{P}(\ginf)\ge \dual(\ginf)$.
\end{proposition}

\begin{proof}
  This is a particular case of Proposition \ref{prop:2} but can also
  be checked directly: for all $\mu\in \mathscr{P}_{\log}(\C )$ such
  that $ \int \log|Q_{n}| \, \dd \mu \ge 0$ for all $n$ and $(a_{n})$
  as in (\ref{eq:6}) we have
  \begin{displaymath}
    \int \ginf \, \dd \mu \ge     \int \Big( \ginf 
    - \sum_{n\in \mathbb{N}} a_{n}\log |Q_{n}|  \Big) \, \dd \mu
\ge \inf_{z\in \C } \Big( \ginf(z) - \sum_{n\in \mathbb{N}} a_{n}\log |Q_{n}(z)| \Big),     
\end{displaymath}
which gives the inequality.
\end{proof}

Our main result in this section shows  that these problems
satisfy the strong duality property, which consists in the equality
between their optimal values.  Its proof is modelled in that of
Theorem~\ref{thm:1} for the finite dimensional case. 

\begin{theorem} \label{strong_duality} We have 
  $ \dual(\ginf)=\primal(\ginf) \in \R $. 
\end{theorem}
 
\begin{proof}
  Let $\csP_{\log}'(\C )$ be the convex set of probability
  measures on $\C $ that integrate all the functions of the
  form $\log|Q_{i}|$. For $n\in \mathbb{N}$
  consider the convex subset of $\R ^{n+2}$  defined as
  \begin{displaymath}
V_{n}=\Big\{ \Big( \int \ginf \,  \dd\mu ,  \int \log|Q_1| \,\dd\mu ,
    \dots, \int \log|Q_n| \,\dd\mu\Big) \, \Big| \
    \mu \in \csP_{\log}'(\C ) \Big\},
  \end{displaymath}
  and for each $\lambda\in \R $ consider also the convex subset
  of $\R ^{n+2}$ defined as
  \begin{displaymath}
    W_{n,\lambda}=\{ (t,x_{1},\dots, x_{n}) \, | \ t\le \lambda, \ x_{1}, \dots, x_{n}\ge 0\}.
  \end{displaymath}
We have that
  $V_{n}\cap W_{n,\lambda}\ne \emptyset$ if and only if there exists
  $\mu \in \mathscr{P}_{\log}'(\C )$ such that
  $ \int \ginf \,\dd\mu \le \lambda $ and
  $ \int \log|Q_{i}| \,\dd\mu \ge 0$, $i=1,\dots, n$.
  Now set
  \begin{displaymath}
    \lambda_{n}=\inf\{ \lambda \in \R  \, | \
V_{n}\cap W_{n,\lambda}\ne \emptyset\}.
\end{displaymath}
  Since $\ginf$ is bounded from below the values
  $\int \ginf\, \dd \mu$ for
  $\mu\in \mathscr{P}_{\log}(\C )$ are also bounded from below,
  and so $\lambda _{n}>-\infty$.
  On the other hand, the Dirac delta
  measure $\mu=\delta_{z}$ for $z\gg 0$ satisfies
  \begin{equation}
    \label{eq:11}
    \int \log|Q_{i}| \, \dd \mu = \log|Q_{i}(z)| > 0, \quad i=1,\dots, n.
  \end{equation}
In particular $\lambda _{n}\le \ginf(z) <+\infty$. Hence $\lambda
_{n}\in \R $. 

  Since $\lambda_{n}-1/n<\lambda _{n}$ the convex subsets $V_n$ and
  $W_{n,\lambda_{n}-1/n}$ are disjoint, and so by the hyperplane
  separation theorem there exists
  $h\in (\R ^{n+2})^{\vee} \simeq \R ^{n+2}$ with
  $h\ne 0$ such that $h(p)\ge h(p')$ for all $p\in V_{n}$ and
  $p'\in W_{n,\lambda_{n}-1/n}$.  Writing
  $h=(b,-a_{1},\dots, -a_{n})$ with $b,a_{i}\in \R $
  these conditions amount~to
  \begin{equation}
    \label{eq:201}
    b \int \ginf \,\dd\mu  - \sum_{i=1}^{n}a_{i}\int \log|Q_{i}|
    \,\dd\mu \ge     b\, t - \sum_{i=1}^{n}a_{i}x_{i}  
  \end{equation}
  for all $ \mu \in \csP_{\log}'(\C )$,
  $t\le \lambda_{n}-1/n$ and $x_{1},\dots, x_{n} \ge0$.

  Since $x_{i}$ can be arbitrarily large this inequality implies that
  $a_{i}\ge 0$, $i=1,\dots, n$, and since $t$ can also
  be arbitrarily negative we similarly deduce that $b\ge 0$.  To exclude
  the possibility that $b=0$, consider the case $x_{i}=0$,
  $i=1,\dots,n$,  and note that in this situation
  \eqref{eq:201} gives
  \begin{displaymath}
    -\sum_{i=1}^{n}a_{i}\int \log|Q_{i}| \,\dd\mu \ge 0    \quad \text{ for all }
    \mu \in \csP_{\log}'(\C ).
  \end{displaymath}
  Since $h\not = 0$ we have that $a_{i}>0$ for some
  $i\in \{1,\dots, n\}$, in which case the inequality does not hold
  for the measure in \eqref{eq:11}.  We conclude that $b>0$, and so we
  can assume without loss of generality that $b=1$.

  Setting $t=\lambda_{n}-1/n$  and $x_{1}=\cdots=x_{n}=0$ in~(\ref{eq:201}) we get
  \begin{equation}
\label{eq:401}
\int \ginf \,\dd\mu  - \sum_{i=1}^n a_{i} \int \log|Q_i| \,\dd \mu \ge \lambda_{n} -\frac{1}{n} \quad \text{ for all } \mu\in \mathscr{P}'_{\log}(\C ).
 \end{equation}
Considering this inequality  for 
$\mu=\delta_{z}$ with
$z\in \C  \setminus \overline{\mathbb{Q}}$ we get
\begin{displaymath}
\ginf(z) - \sum_{i=1}^n a_{i}\log|Q_{i}(z)| \ge \lambda_n
-\frac1n  \quad \text{ for all } z\in \C  \setminus
\overline{\mathbb{Q}}, 
\end{displaymath}
which extends by density to all $z\in \C $. Hence
$\dual(\ginf)\ge \lambda _{n}-1/n$. Therefore 
\begin{equation}\label{eq:3}
  \limsup_{n \rightarrow \infty} \lambda_n \leq \dual(\ginf).
\end{equation}

On the other hand we can choose
$\mu_{n} \in \mathscr{P}_{\log}'(\C )$ satisfying
  \begin{equation}
  \label{eq:101}
    \int \ginf \,\dd\mu_{n} \le \lambda_{n}+\frac{1}{n} \quad \text{ and } \quad
 \int \log|Q_{i}| \,\dd\mu_{n} \ge 0 , \quad  i=1,\dots, n.
\end{equation}
By \eqref{eq:3}
\begin{displaymath}
  \limsup_{n\to \infty} \int \ginf \,\dd\mu_{n}\leq \dual(\ginf).
\end{displaymath}
  By Proposition \ref{truncation-log infinity}, for any
  $\varepsilon>0$ there exists $R>1$ such that
\begin{displaymath}
  \limsup_{n\to \infty} \int \ginf \,\dd\musw_{n,R}\leq \dual(\ginf)+\varepsilon,
\end{displaymath}
where $\musw_{n,R}$ denotes the sweetened truncation of $\mu_{n}$
(Definition~\ref{def truncation}).  Since all the probability measures
$ \musw_{n,R}$ are supported on the ball $\DR=\{ |z|\leq R\}$, up to
taking a subsequence we can assume that they converge properly to a
probability measure~$\mu_{R}$ with support contained in~$\DR$. Then
\begin{flalign*}
	\int \ginf \,\dd\mu_R = \lim_{n\rightarrow \infty} \int \ginf \,\dd \musw_{n,R}\leq  \dual(\ginf)+\varepsilon, 
\end{flalign*} 
and  by Lemma \ref{monotone} and
Proposition \ref{truncation
  properties}\eqref{Smith_truncation_keeps_Smith_conditions} we also have
\begin{flalign*}
  \int \log|Q_{m}| \,\dd\mu_{R} \geq \liminf_{n \rightarrow \infty} \int \log|Q_m| \,\dd\musw_{n,R} \geq 0 \quad
  \text{ for all } m.
\end{flalign*}
Hence $\mu _{R}$ is a candidate for the primal problem and we
deduce that $\primal(\ginf) \leq \dual(\ginf)+\varepsilon$.  Taking
$\varepsilon> 0$ arbitrarily small we get
$\primal(\ginf) \leq \dual(\ginf)$, and we obtain the equality by combining this
with the weak duality property (Proposition~\ref{prop:1}).

The fact that 
$\primal(\ginf)$ is a real number follows from  the estimates
\begin{displaymath}
  -\infty < \inf_{z\in \C} \ginf(z) \le \primal(\ginf) \le
  \int \ginf\,\dd \lambda _{S_{1}} < \infty.
\end{displaymath}
\end{proof}

\begin{remark}
  \label{rem:101}
A related result was proved by Smith for functions on subsets of the
real line with super-logarithmic behavior around the point at infinity
\cite[Theorem~5.11]{Smith}. We extend this result to the boundary
case when the behavior is logarithmic,  
by slightly modifying the definition of sweetened truncation and using the
estimate from Proposition \ref{truncation-log infinity}.
\end{remark}

\begin{remark}
  \label{rem:6}
  As we will see in Section \ref{Nuno}, there may not exist a measure
  realizing the optimal value $\primal(\ginf)$. On the other hand,
  when such measure does exist it is not necessarily unique. An
  example is given by $\ginf(z) = \log^+|z/2|$: for all $R \in [1,2]$
  the equilibrium measure  $\lambda_{S_{R}}$ lies in
  $\mathscr{P}_{\log}^{\overline{\Z }}(\C )$ and attains the optimal
  value of $\primal(\ginf)$.
\end{remark}

  \begin{remark}
    More generally, when $\ginf(z)$ is \emph{radial} (that is it only
    depends on $|z|$) and subharmonic, the theory of toric varieties
    applies, see for instance \cite{Burgos_book_2014}. In particular
    \cite[Corollary 3.10]{BPSminima} implies that the measure
    $\lambda_{S_{1}}$ attains the optimal value of $\primal(\ginf)$,
    and if $\ginf$ is strictly subharmonic then \cite[Theorem
    4.18]{BPRStoric} implies that $\lambda _{S_{1}}$ is the unique
    measure attaining the optimal value.

    Since we are in dimension one it is simpler to give a direct proof
    of these facts.  Since $\ginf$ is radial and subharmonic we have that
    $\varphi(t) \coloneqq \ginf(e^t)$ is convex. Note also that since
    $\lim_{t \rightarrow -\infty} \varphi(t)=\ginf(0)$ is finite,
    $\varphi$ is non-decreasing. Then
 	\begin{displaymath}
		\int \ginf(z) \,\mathrm{d}\mu = \int \varphi \big( \log|z| \big) \,\mathrm{d}\mu \geq \varphi \left( \int \log|z| \,\mathrm{d}\mu \right) \geq \varphi(0) = \ginf(1) = \int \inf(z) \,\mathrm{d} \lambda_{S_{1}},
              \end{displaymath}
              where the first inequality is Jensen's inequality, whereas
              the second uses that $\varphi$ is
              non-decreasing and that $\int \log|z| \,\dd \mu
              \ge 0$.
              
              If we assume further that $\ginf$ is strictly
              subharmonic then $\lambda_{S_{1}}$ is the unique measure
              in
              $\mathscr{P}_{\log}^{\overline{\mathbb{Z}}}(\mathbb{C})$
                attaining the optimal value.  In fact, in this case
                $\varphi$ is strictly convex and increasing. If
                $\mu \in \mathscr{P}_{\log}^{\overline{\Z }}(\C )$ is
                a minimizer, then Jensen's inequality
\begin{displaymath}
	\int \varphi \big( \log|z| \big) \,\mathrm{d}\mu \geq \varphi \left( \int \log|z| \,\mathrm{d}\mu \right)
      \end{displaymath} must be an equality, which means $|z|$ must be constant on the support of $\mu$, say $|z|=r$. Then
\begin{displaymath}
	\int \varphi \big( \log|z| \big) \,\mathrm{d}\mu = \varphi \left( \int \log|z| \,\mathrm{d}\mu \right) = \varphi(\log r).
      \end{displaymath} Now $\mu$ is a minimizer, so $\varphi(\log r)=\varphi(0)$ and hence $r=1$.  Then $\mu \in \mathscr{P}_{\log}^{\overline{\Z }}(\C )$ is supported on the unit circle and so $\mu=\lambda_{S_{1}}$. 
%  We will not give details here, but these assertions can be proved based on the study of toric Arakelov geometry established in \cite{Burgos_book_2014, BPSminima, BPRStoric}.
\end{remark}

%%% Local Variables:
%%% mode: latex
%%% TeX-PDF-mode: t
%%% TeX-source-correlate-mode: t
%%% TeX-master: "main"
%%% End: 

\section{The essential minimum of height functions} \label{heights}

In this section we briefly recall the Arakelov point of view of our height
functions and use the strong duality theorem to extract some
properties for their essential minima.

\label{sec:essent-minim-heights}
\subsection{Arakelov theory on
  $\mathbb{P}^1$} \label{Arakelov_theory_on_P^1} Let $\mathbb{P}^{1}$
be the projective line over $\mathbb{Q}$ equipped with its universal
line bundle $\mathcal{O}(1)$. Consider their canonical integral
models $\PP^{1}_{\Z}$ and $\mathcal{O}(1)_{\mathbb{Z}}$ together with
a continuous metric $\|\cdot\|$ on the holomorphic line bundle
$\cO(1)_{\C}$ over the Riemann sphere $\PP^1(\C)$ that is
invariant under the complex conjugation.  The pair
$\overline{\mathcal{O}(1)_{\mathbb{Z}}} =
(\mathcal{O}(1)_{\mathbb{Z}},\|\cdot\|)$ is called a \emph{metrized
  line bundle} on $\PP_{\Z}^1$, and following \cite[Section
3]{BostGilletSoule} it induces a height
function
\begin{displaymath}
  \Ht_{\overline{\mathcal{O}(1)_{\mathbb{Z}}}}\colon \mathbb{P}^{1}(\overline{\mathbb{Q}}) \longrightarrow \mathbb{R}.
\end{displaymath}

% For any $n\geq 1$, the hermitian metric uniquely
% extends to the tensor power $\mathcal{O}(1)^{\otimes n}
% =\mathcal{O}(n)$. In particular, for any $s\in
% H^0(\mathbb{P}^1_\Z ,\mathcal{O}(n))$, we can consider the sup
% norm  
% $$\|s\|_\infty\coloneqq \sup \Big\{\|s(x)\| \,\Big|\, x \in \mathbb{P}^1(\C) \Big\}.$$

Let $(x_{0}\colon x_{1})$ be the homogeneous coordinates of
$\PP^{1}$. Then $x_{1}$ is a global section of $\cO(1)_{\mathbb{C}}$
vanishing at the point at infinity $\infty=(1:0)$. Using the
identification $\PP^1(\C)\backslash (1:0) \simeq \mathbb{C} $ given by
the map $(x_0\colon x_1) \longmapsto z=x_0/x_1$, the metric $\|\cdot\|$ is
encoded by its associated \emph{Green function of continuous type}
  \begin{equation}
    \label{eq:9}
    \ginf \colon \C \longrightarrow \R , \quad z
    \longmapsto -\log\|x_1(z\colon1)\|. 
  \end{equation}
  This is a continuous function that is invariant under the complex
  conjugation and satisfies the asymptotics
  $ \ginf(z)=\log|z|+a+ o(1) $ as $z\to \infty$ with
  $a\in \mathbb{C}$. Conversely, every Green function of continuous type
  defines a continuous and conjugation invariant metric on
  $\cO(1)_{\mathbb{C}}$.
  
  Already in the work of Faltings leading to the proof of Mordell's
  conjecture \cite{Faltings_Finiteness}, it became apparent that for
  many arithmetic applications one should consider metrics admiting
  some singularities. In the recent work \cite{Yuan-Zhang}, Yuan and
  Zhang introduced adelic line bundles on quasi-projective varieties
  allowing to extend Arakelov geometry to a wide class of singular
  metrics.
  
In particular, this allows to consider an adelic line
  bundle $\overline{L}$ on the affine line
  $\mathbb{A}^1_{\mathbb{Q}}$ that is defined by the line bundle
  $\mathcal{O}(1)_{\mathbb{Z}}$ on the compactification
  $\mathbb{A}^{1}_{\mathbb{Z}} \subseteq \mathbb{P}^{1}_{\mathbb{Z}}$,
  together with a metric $\|\cdot \|$ on $\cO(1)_{\C}$ that can be
  singular at the point at infinity. Following Section 5.3.1 in \emph{loc. cit.} we
  can associate to this adelic line bundle a height function
  \begin{displaymath}
  \Ht_{\overline{L}}\colon \mathbb{A}^{1}(\overline{\mathbb{Q}}) = \overline{\mathbb{Q}} \longrightarrow \mathbb{R}.
\end{displaymath}
The archimedean Green functions associated to these particular adelic line bundles
are exactly the Green functions in the sense of Definition \ref{log
  singular} \cite[Theorem 3.6.4]{Yuan-Zhang}.  Within this family of line bundles, $\overline{L}$ is
characterized by its Green function $\ginf$, so for simplicity we denote
its height function by $ \Ht_{\ginf }$.

This height function can be described in very concrete terms. Let
$\alpha \in \overline{\Q}$ and denote by $P_{\alpha }\in \Z[x]$
a primitive irreducible polynomial with $P_{\alpha }(\alpha )=0$, which
is unique up to a sign.  Let~$c_{\alpha }$ be the leading coefficient
of $P_{\alpha }$ and $O(\alpha )\subseteq \overline{\Q}$ its set of
zeros. The height of $\alpha $ is then given by
\begin{equation}\label{eq:4}
  \Ht_{\ginf}(\alpha )=\frac{1}{\deg (\alpha )}\left( \log|c_{\alpha }|
      + \sum_{\beta  \in O(\alpha )}\ginf(\beta )\right).
  \end{equation}
  If $\alpha \in \overline{\Z}$ then we can choose $P_{\alpha }$
monic and so
\begin{equation}
  \label{eq:10}
  \Ht_{\ginf}(\alpha )= \frac{1}{\deg(\alpha )}\sum_{\beta  \in
    O(\alpha )}\ginf(\beta ),
\end{equation}
whereas for an arbitrary $\alpha\in \overline{\mathbb{Q}}$ we only have
  the inequality
\begin{displaymath}
  \Ht_{\ginf}(\alpha )\ge \frac{1}{\deg(\alpha )}\sum_{\beta  \in
    O(\alpha)}\ginf(\beta ).
\end{displaymath}

There is another description of the height of an algebraic number as a
sum of local contributions that will be useful later. 
Let $\cM_{\Q}$ be the set of places of $\Q$, that is
\begin{displaymath}
  \cM_{\Q}=\{p\in \Z \mid p>0\ \text{prime}\}\cup \{\infty\}.
\end{displaymath}
For each $\nu\in \cM_{\Q}$ let $|\cdot|_{\nu}$ denote either the usual
absolute value of $\Q$ if $\nu =\infty$ or the $p$-adic absolute value
of $\mathbb{Q}$ if $\nu=p$ for some prime $p$. This absolute value
extends uniquely to the complete and algebraically closed field
$\C_{\nu}$. Then the
height of $\alpha $ can be alternatively written as~\cite[Section
1.5.7]{BombieriGubler}
\begin{equation}
  \label{eq:8}
  \Ht_{\ginf}(\alpha )=\frac{1}{\deg(\alpha )}\left( \left(
      \sum_{p \textrm{ prime}} \sum_{\beta \in O(\alpha )}
      \log^{+}|\beta |_{p}\right)
      + \sum_{\beta \in O(\alpha)}\ginf(\beta )\right).
  \end{equation}

\begin{definition}
  Let
  $\Ht_{\ginf}\colon \overline{\mathbb{Q}}\longrightarrow \mathbb{R}$
  be the height function associated to a Green function $\ginf$. Its 
  \emph{essential minimum} is defined as
  \begin{displaymath}
\ess(\Ht_{\ginf})\coloneqq\inf \left\{ \liminf_{n \rightarrow \infty}
  \Ht_{\ginf}(\alpha_n) \,\middle|\, (\alpha_n) \textrm{ is a  sequence of distinct algebraic numbers}
\right\}.
\end{displaymath}
\end{definition}

\begin{remark}
  To stay as elementary as possible, from now on we will forget about
  the Arakelov point of view and focus on the Green function $\ginf$
  as the source for the height function. Nevertheless this point of
  view suggests many generalizations of the present work. First,
  instead of considering the canonical model
  $(\PP^{1}_{\Z},\mathcal{O}(1)_{\mathbb{Z}})$ of the pair
  $(\PP^{1},\mathcal{O}(1))$ one can consider arbitrary integral
  models, or even more generally one can consider general adelic line
  bundles with different metrics over each place of $\Q$. Second,
  instead of the divisor $[(1:0)]$ of $\PP^{1}$ with respect to which
  the Green function is defined, one can consider arbitrary
  divisors. Finally, instead of $\PP^{1}$ one can consider arbitrary
  curves or even higher dimensional varieties.
\end{remark}

\subsection{The essential minimum and linear programming}
\label{sec:essent-minim-atta}
In this section we relate the essential minimum of the height function
with the linear programming problems $\cD(\ginf)$ and~$\cP(\ginf)$.

\begin{proposition}\label{prop:3} The relation
    $ \ess(\Ht_{\ginf})\ge \cD(\ginf) $
  holds true.
\end{proposition}
\begin{proof}
  Recall the product formula
  \begin{math}
    \prod_{\nu \in \cM_{\Q}} |\alpha |_{\nu}=1
  \end{math}
  for $ \alpha \in \Q^\times$,
  which implies that
  \begin{equation}
    \label{eq:12}
    \sum_{\nu\in \cM_{\Q}}\sum_{\beta \in O(\alpha)}\log |\beta |_{\nu}=0 \quad \text{ for }  \alpha \in \overline \Q^\times
  \end{equation}
because $\prod_{\beta \in O(\alpha )} \beta \in \Q$.
Let $Q_{i}$, $i=1,\dots, n$, be nonzero polynomials with integer coefficients
and $a_{1},\dots,a_{n}\in \R_{\ge 0}$. If we show that 
\begin{equation}\label{eq:15}
  \ess(\Ht_{\ginf})\ge \inf_{z\in \C} \left( \ginf(z)-\sum_{i=1}^{n}a_{i}\log|Q_{i}(z)|\right),
\end{equation}
then the proposition will be proved.

If $\sum a_{i}\deg(Q_{i})>1$ then
\begin{displaymath}
  \inf_{z\in \C} \left( \ginf(z)-\sum_{i=1}^{n}a_{i}\log|Q_{i}(z)|\right)=-\infty
\end{displaymath}
and the inequality  \eqref{eq:15} is trivially true. So we can assume that 
\begin{equation}
  \label{eq:16}
  \sum a_{i}\deg (Q_{i})\le 1.
\end{equation}

Let $Z \subseteq \overline{\mathbb{Q}} $ be the union of the set of
zeros of the $Q_{i}$'s.  By the product formula \eqref{eq:12}, for all 
$\alpha\in \overline \Q\setminus Z$ we have that
\begin{multline*}
  \Ht_{\ginf}(\alpha)=\frac{1}{\deg(\alpha)}
  \sum_{p \textrm{ prime}} \sum_{\beta \in O_{\nu}(\alpha)}
  \left(\log^{+}|\beta|_{p}- \sum_{i=1}^{n}a_{i}\log|Q_{i}(\beta)|_{p}\right)\\
+\frac{1}{\deg(\alpha)}\sum_{\beta \in O_{\nu}(\alpha)}\left(\ginf(\beta)-\sum_{i=1}^{n}a_{i}\log|Q_{i}(\beta)|\right).
\end{multline*}
By \eqref{eq:16} and the fact that the polynomials $Q_{i}$ have
integer coefficients,  for every prime $p$ we have that
\begin{displaymath}
  \sum_{i=1}^{n}a_{i}\log|Q_{i}(\beta)|_{p}\le \log^{+}|\beta|_{p}.
\end{displaymath}
Therefore 
\begin{displaymath}
  \Ht_{\ginf}(\alpha)\ge
  \frac{1}{\deg(\alpha)}\sum_{\beta \in O(\alpha)}\left(\ginf(\beta)-\sum_{i=1}^{n}a_{i}\log|Q_{i}(\beta)|\right) \quad \text{ for all }
  \alpha\in \overline
\Q\setminus Z.
\end{displaymath}
This givens the inequality \eqref{eq:15} and proves the proposition.
\end{proof}

\begin{proposition}\label{prop:4}
  The relation $\ess(\Ht_{\ginf})\le \cP(\ginf)$
  holds true.
\end{proposition}
\begin{proof}
  Let $\mu \in \mathscr{P}_{\log}^{\overline{\Z }}$ and take a
  sequence of distinct algebraic integers $(\alpha_n)$ such that
  $O(\alpha_n)$ converges log-weakly to $\mu$. Using formula
  \eqref{eq:10} and the fact that $\ginf$ is a continuous function of
  logarithmic growth, we have that
  \begin{equation}
    \label{eq:32}
    \lim_{n\to \infty} \Ht_{\ginf}(\alpha_{n})=
    \lim_{n\to \infty} \frac{1}{\deg(\alpha_n)}\sum_{p \in O(\alpha_n)} \ginf(p)
    =\int \ginf \,\dd\mu.
  \end{equation}
  Hence
  $\displaystyle{ \ess(\Ht_{\ginf})\leq \inf \Big\{ \int \ginf
    \,\dd\mu \, \Big| \ \mu \in \mathscr{P}^{\overline{\Z }}_{\log}(\C
    ) \Big\}= \primal(\ginf)}$, as stated.
\end{proof}

% Let $\overline{\mathcal{L}} = (\mathcal{O}(1),\|\cdot\|)$ be a hermitian line bundle with sub-$\log$-singularity metric at $\infty$.  We define the \emph{asymptotic maximal slope} of $\overline{\mathcal{L}} $ by $$\hat{\mu}(\overline{\mathcal{L}})\coloneqq\sup_{n\geq 1\atop s \in H^0(\mathbb{P}^1_\Z ,\mathcal{O}(n))} - \frac{\log\|s\|_\infty}{n}.$$

Putting together Propositions \ref{prop:3} and \ref{prop:4} with
Theorem \ref{strong_duality} we obtain the main theorem of this
section. 

\begin{theorem} \label{strong_duality_essential_minimum}
Let $\ginf$ be a Green function. Then $\cD(\ginf) = \ess(\Ht_{\ginf}) = \cP(\ginf)$.
\end{theorem}

A direct consequence of Theorem \ref{strong_duality_essential_minimum}
is that the essential minimum can be reached by a sequence of
algebraic integers.

\begin{corollary} \label{ess_attained_by_integer}
	There exists a sequence of distinct algebraic integers
        $(\alpha_n)$ such that
        \begin{equation}
          \label{eq:33}
         \lim_{n\to \infty}
        \Ht_{\ginf}(\alpha_n) = \ess(\Ht_{\ginf}). 
        \end{equation}
\end{corollary}

As a consequence of a result of Szachniewicz based on the theory of
globally valued fields~\cite[Theorem A]{szachniewicz2023} it can be
shown that $\Ht_{\ginf}(\overline{\Q})$ is dense in the interval
$[\ess(h_{\ginf}),\infty)$.  As another application of our results we show that
this is already the case for the set of heights of algebraic integers. 

\begin{corollary}\label{spectrum}
The set
$\Ht_\ginf(\overline{\Z})$ is dense in $[\ess(h_{\ginf}),\infty)$.  
\end{corollary}

\begin{proof} 
  By \eqref{eq:32} the set
  \begin{displaymath}
  I=\left\{\int \ginf \, \dd \mu \,\middle|\, \mu \in \mathscr{P}_{\log}^{\overline{\Z }}(\C ) \right\}  
  \end{displaymath}
  is contained in the closure of $\Ht_\ginf(\overline{\Z})$. By
  Theorem \ref{SOS non compact} it is a convex subset of $\R$, hence
  an interval.  By Theorem \ref{strong_duality_essential_minimum} this
  interval contains real numbers that are arbitrarily close to
  $\ess(h_{\ginf})$, and considering $\mu=\lambda_{S_{R}}$ for
  $R\gg 0$ we can also see that it is not bounded above.  Hence
  $(\ess(h_{\ginf}), \infty) \subseteq I$, which gives the statement.
\end{proof}

\begin{remark}
  \label{rem:7}
  A direct consequence of Corollary \ref{spectrum} is that the
  sequence of distinct algebraic integers $(\alpha_{n})$ in
  \eqref{eq:33} can be chosen so that $h_{\ginf}(\alpha_{n})$ is
  monotonically decreasing.
\end{remark}

\subsection{Asymptotic maximal slope}
\label{sec:maximal-slope}

In this section we prove that in the dual problem we can reduce to
linear combinations with rational  coefficients. This has
two applications. The first one is to relate the essential minimum with another
Arakelov theory invariant called the asymptotic maximal slope and the
second one is the left computability of the essential minimum as we
will see in section \ref{sec:comp-essent-minim}. Consider
the rational dual problems
\begin{displaymath}
  \dual_{\Q}(\ginf) \coloneqq \sup \left\{ \inf_{z\in \C} \Big(
    \ginf(z) - \sum_{i=1}^k a_i \log |Q_i(z)| \Big) \,\middle|\,
    a_i\in \Q_{\geq 0}, Q_i\in \Z[x]\setminus\{0\},\sum a_{i}\deg(Q_{i})\le 1
  \right\}.
\end{displaymath}
and 
\begin{displaymath}
  \dual'_{\Q}(\ginf) \coloneqq \sup \left\{ \inf_{z\in \C} \Big(
    \ginf(z) - \sum_{i=1}^k a_i \log |Q_i(z)| \Big) \,\middle|\,
    a_i\in \Q_{\geq 0}, Q_i\in \Z[x]\setminus\{0\},\sum a_{i}\deg(Q_{i})<1
  \right\}.
\end{displaymath}
Clearly $\dual(\ginf)\ge \dual_{\Q}(\ginf)\ge \dual'_{\Q}(\ginf)$.
\begin{proposition}\label{prop:5}
  The equalities $\dual(\ginf)= \dual_{\Q}(\ginf)=\dual'_{\Q}(\ginf)$ hold. 
\end{proposition}
\begin{proof}
Let $n\geq 1$ and take ${\bf a}=(a_1,\ldots,a_n)\in \R_{\geq 0}^n$ and
$(Q_1,\ldots,Q_n) \in (\Z [x]\setminus\{0\})^n$.   Let 
	\[
	\varphi_{\bf a}(z) \coloneqq \ginf(z) - \sum_{i=1}^n a_i \log|Q_i(z)|. 
	\]

Then, it is enough to show that for any $\varepsilon>0$ and any ${\bf a}\in \R_{\geq 0}^n$, there exists ${\bf b}\in \Q_{\geq 0}^n$ with $\sum_{i=1}^n b_i\deg Q_i<1$ such that
\begin{equation}\label{goal}
\inf \varphi_{\bf a} \leq \inf \varphi_{\bf b} +\varepsilon.
\end{equation}

Relation \eqref{goal} holds trivially if $\inf \varphi_{\bf a}=-\infty$ or ${\bf a}={\bf 0} \in \R^n$, so we can assume $\inf \varphi_{\bf a}>-\infty$ and $a_i\neq 0$ for all $i$. Then, necessarily 
\begin{equation}\label{politopo}
\sum_{i=1}^na_i\deg Q_i\leq 1.
\end{equation} 
%Set $B=\min\{a_i/2 : i=1,\ldots,n\}$.
 Let $U$ be an open and bounded neighborhood of the set of zeroes of
 $\prod_{i=1}^n Q_i(x)$  such that for all  $z\in U$, it holds  
 \begin{equation}\label{explosion}
\inf   \varphi_{\bf a}\leq
\ginf(z)-\sum_{i=1}^na_i\max\left\{\frac{1}{2}\log|Q_i(z)|,
  \log|Q_i(z)|\right\}
 \end{equation}

Choose $R\geq 1$ big enough so that $\log|Q_i(z)|\geq 0$  for all
$|z|\geq R$ and $i=1,\ldots,n.$ Set $K \coloneqq U^c\cap \{|z|\leq
R\}$.  Since $K$ is a compact set, there exist ${\bf b}\in \Q_{\geq
  0}^n$ such that  
\begin{displaymath}
  \frac{a_i}{2}\leq b_i < a_i \textrm{ for all } i  \textrm{ and  }\inf_{x\in K} \varphi_{\bf b}(x)\geq  \inf_{x\in K} \varphi_{\bf a}(x) - \varepsilon.
\end{displaymath}

In particular, in view of \eqref{politopo} we have that

\begin{equation}\label{compact}
\sum_{i=1}^nb_i\deg Q_i<1 \textrm{ and } \inf_{z\in K} \varphi_{\bf
  b}(z)\geq \inf  \varphi_{\bf a} - \varepsilon.
\end{equation}
On the other hand, we claim that
\begin{equation}\label{inf U}
 \inf_{z\in U}\varphi_{\bf b }(z)\geq \inf\varphi_{\bf a}.
 \end{equation}
Indeed,  due to the choice of the $b_i$ we have
\begin{displaymath}
  \varphi_{\bf b }(z)\geq \ginf(z)
  -\sum_{i=1}^na_i\max\left\{\frac{1}{2}\log|Q_i(z)|,
    \log|Q_i(z)|\right\}, 
\end{displaymath}
so the claim follows from  \eqref{explosion}.  Finally, for $z$ such that $|z|\geq R$ we have 

\begin{displaymath}
  \varphi_{\bf b }(z)-\varphi_{\bf a}(z)=\sum_{i=1}^n(a_i-b_i)\log|Q_i(z)|\geq 0,
\end{displaymath}
thus $\inf_{|z|\geq R} \varphi_{\bf b}(z)\geq \inf
\varphi_{\bf a}.$ This estimate, together with \eqref{compact} and
\eqref{inf U}, imply \eqref{goal}.   
\end{proof}

We now consider on $\PP^{1}_{\Z}$ the line bundles $\cO(n)=\cO(1)^{\otimes
n}$. The metric on $\cO(1)$ induces metrics on each $\cO(n)$. For each
global section $s\in \rH^{0}(\PP^{1}_{\Z},\cO(n))$ the sup norm is defined
as
\begin{displaymath}
  \|s\|_{\infty} \coloneqq \sup_{z\in \PP^{1}(\C)}\|s(z)\|.
\end{displaymath}
For the particular family of adelic line bundles on $\mathbb{A}^1_\mathbb{Q}$ introduced in Section \ref{Arakelov_theory_on_P^1}, the \emph{asymptotic maximal slope} of $\overline{L}$ is defined as
by 
\begin{displaymath}
  \hat{\mu}(\overline{L})\coloneqq\sup_{\substack{n\geq 1\\s \in
    \rH^0(\PP_{\Z}^1,\mathcal{O}(n)) \backslash \{0\}}} -
  \frac{\log\|s\|_\infty}{n}. 
\end{displaymath}
By unfolding the definitions, one can see that
\begin{displaymath}
  \hat{\mu}(\overline{L})=\dual_{\Q}(\ginf).
\end{displaymath}
Hence, as a direct consequence of Theorem \ref{strong_duality_essential_minimum} and Proposition \ref{prop:5} we obtain the next result.

\begin{corollary}\label{cor:1}
  The equality
  $\hat{\mu}(\overline{L})=\ess(\Ht_{\overline{L}})$ is satisfied.
\end{corollary}

\begin{remark} \label{rem:5} When $\ginf$ is subharmonic and of continuous type, this equality is a particular case of a result of Balla\"y
  \cite[Theorem 1.2]{Ballay}. Yuan conjectured
  a more general version of the result in \emph{loc. cit.}, see
  \cite[Conjecture 5.3.5]{Yuan-Zhang}. This conjecture  has been
  established in some cases beyond \cite[Theorem 1.2]{Ballay}, see
  \cite[Theorem 1.8]{Qu-Yin} and  \cite[Theorem 5.3.6]{Yuan-Zhang}. 
\end{remark}

\subsection{An example with no minimizer} \label{Nuno}

In this subsection we present an example due to Nuno Hultberg, where the infimum \[
	\inf_{\mu \in \mathscr{P}_{\log}^{\overline{\Z }}(\C )} \int \ginf \,\dd\mu
	\] is not attained by any measure $\mu_0 \in
        \mathscr{P}_{\log}^{\overline{\Z }}(\C )$.
Consider the automorphism $f \colon \mathbb{P}^1 \longrightarrow
\mathbb{P}^1$, $x \longmapsto x^{-1}+2$. In homogeneous coordinates is
given by $(x_0\colon x_1) \longmapsto (x_1+2x_0\colon x_0)$ under
the identification $(x:1) = (x_0\colon x_1)$.  Let
$\overline{\mathcal{L}}=\big(
\mathcal{O}(1),\|\cdot\|_{\operatorname{can}} \big)$, where the
canonical metric $\|\cdot\|_{\operatorname{can}}$ is the metric whose
Green function is $\log^{+}$.
Consider the hermitian line bundle
 $f^*\overline{\mathcal{L}}$. By the projection formula
 $\Ht_{f^*\overline{\mathcal{L}}}=\Ht_{\overline{\mathcal{L}}}\circ
 f$, we see that
 \begin{displaymath}
   \ess(\Ht_{f^*\overline{\mathcal{L}}})=\ess
 \Ht_{\overline{\mathcal{L}}}  =\nobreak 0.
\end{displaymath}
The Green function associated with $f^{\ast}\overline{\cL}$ is
$\ginf(z) = \log^+|z^{-1}+2| + \log|z|$. In particular, Theorem
\ref{strong_duality_essential_minimum} applies to
$\Ht_{f^*\overline{\mathcal{L}}}$. 
	
Assume there exists $\mu_0 \in
\mathscr{P}_{\log}^{\overline{\Z }}(\C )$ such that $
\int \ginf \,\dd\mu_0 = \ess(\Ht_{\ginf})= \nobreak 0$. Then
there exists a sequence of distinct algebraic integers $(\alpha _n)$ such that
$\delta _{O(\alpha _n)}$ converges to $\mu_0$ in the $\log$-weak sense. Therefore $
\Ht_{\ginf}(\alpha _n) $ converges to $\int \ginf
\,\dd\mu_0 =0$. Using the projection formula and Bilu's
equidistribution theorem \cite[Theorem 1.1]{Bi97}, we see that
$\delta _{O(\alpha _n)}$ must converge weakly to $f^*\mu_{S^1}$, which forces
$\mu_0=f^*\mu_{S^1}$. However,
\begin{align*}
  \int \log^+|z^{-1}+2| + \log |z| \,\dd f^*\mu_{S^1}(z)
  &= \int \log^+|t| +\log|t-2|^{-1}\dd \mu_{S^1}(t) \\ 
  &=-\int \log |t-2|\dd \mu_{S^1}(t)\\
  &=-\log 2\not = 0,
\end{align*}
which is a contradiction.

The preceeding  argument also shows that $f^*\mu_{S^1} \notin
\mathscr{P}_{\log}^{\overline{\Z }}(\C )$, even though
$\mu_{S^1}$ does belong to $
\mathscr{P}_{\log}^{\overline{\Z }}(\C )$.  To see this
more concretely, consider the polynomial 
\[
P_n(x) = \frac{x^{n+1}-2x^n+1}{x-1} \in \Z [x]
\] and let $\alpha _n$ be a root  of $P_n$. Then $f^{-1}(\alpha _n)$
is an algebraic integer and the sequence $\Ht_{\ginf}(f^{-1}(\alpha _n))
= \Ht_{\overline{\mathcal{L}}}(\alpha _n) $, $n\ge 1$, converges to zero when $n$ goes
to $\infty$. If we look at
the Galois orbit of $f^{-1}(\alpha _n)$, there are $n-1$ conjugates of
$f^{-1}(\alpha _n)$ that are close to the support of $f^*\mu_{S^1}$, yet the
remaining conjugate goes to infinity at the speed of $2^n$. Therefore
$\delta _{O(f^{-1}(\alpha _{n}))}$ converges weakly to $f^*\mu_{S^1}$ but
does not converge $\log$-weakly.
 The situation is similar to Autissier's counterexample
\cite{Autissier} showing that logarithmic equidistribution of small
points is not true in general. 

In this example, it also happens that the essential minimum is strictly smaller than 
\begin{equation}\label{cap1part2}
\inf\left\{ \int \ginf \,\dd\mu_K \,\middle|\, K \textrm{ is a  conjugation invariant compact set of capacity one}\right\}.
\end{equation}

Indeed,
\begin{flalign*}
  \quad \quad \quad
  &
\ginf(z) = \log^+|z^{-1}+2| + \log|z| = \log|z| + \log 2 +
                            o(1) \quad \text{as
                            $|z|\longrightarrow\infty$}. 
	&&
\end{flalign*}
Let $\mu_\infty \coloneqq \frac{1}{2\pi }\Delta \ginf=f^*\mu_{S^1}$. Then
$\ginf(z) = -U^{\mu_\infty} + \log(2)$ and for any probability measure
$\mu$, 
\begin{displaymath}
  \int \ginf \,\dd\mu = \log 2 + \int -U^{\mu_\infty} \,\dd\mu= \log 2
  + \int -U^\mu \,\dd\mu_\infty. 
\end{displaymath}
If $\mu$ is equilibrium measure of a capacity one set, then
$-U^\mu \geq 0$, and we will have $ \int \ginf \,\dd\mu \geq
\log 2$. So \eqref{cap1part2} cannot be the true essential minimum
(which is $0$).

\subsection{The essential minimum of Faltings' height}
\label{sec:essent-minim-falt}
Let $\Gamma=\operatorname{SL}_2(\Z )$ and consider the
associated modular curve $Y$, which is defined over $\Q $. Note
that $Y(\overline{\Q})$ is in bijection with the set of
isomorphism classes of elliptic curves over 
$\overline{\Q}$. 
Let $X$ be the compactification of $Y$. There is a line bundle
$M_{12}$ on $X$ such that $\rH^0(X, M_{12}^{\otimes n}) \simeq
M_{12n}(\Gamma,\Q )$ as Hecke modules, where the latter is the
$\Q $-space of modular forms of
weight $12n$ and level $1$ with rational Fourier coefficients. 

We have a canonical integral model $\big(
\mathcal{X}/\Z ,\mathcal{M}_{12} \big)$ of $\big(
X/\Q $, $M_{12} \big)$. Note that
$\rH^0(\mathcal{X},n\mathcal{M}_{12})\simeq\mathcal{M}_{12n}(\Gamma,\Z )$
as Hecke modules, where the latter is the abelian group of modular
forms of weight $12n$ and level $1$ with integral Fourier
coefficients.

The Petersson metric $\|\cdot\|_{\Pet}$ on $M_{12}(\C )$ is
defined as $\|f\|_{\Pet} (\tau )\coloneqq |f(\tau )| \big( 4 \pi
\operatorname{Im}\tau  \big)^6$, i.e. taking $\big( 4 \pi
\operatorname{Im}\tau  \big)^6$ as metric weight, where $f$ is any modular
form of weight $12$ and level $1$ and $\tau\in \HH\coloneqq \{z\in
\C\mid  \Img(z)>0\}$ is a point in the Poincar\'e upper half
plane. The function  $\|f\|_{\Pet}(\tau )$ is invariant under the
action of $\Gamma $ and descents to a function on $Y(\C)$. 
We
also define the $L_2$-norm on   
$M_{12}(\C )$ by 
\begin{displaymath}
  \| f \|_{L^2}\coloneqq\int_{Y(\C)}
  \|f\|_{\Pet} \,\dd\mu_{\hyp}.
\end{displaymath}

Here, $\mu_{\hyp}$ is the hyperbolic measure on $Y(\C)$ normalized
such that it is a probability measure.  

Let $\overline{\cM_{12}} \coloneqq (\cM_{12},\|\cdot\|_{\Pet})$. Then,
the Faltings height can be brought to our framework by the relation
\cite[Section 2.1]{BurgosGil_essMinFaltings}
\begin{displaymath}
  \Ht_{\rm F}(\alpha )=\frac1{12} \Ht_{\overline{\mathcal{M}_{12}}}(\alpha ).
\end{displaymath}

\begin{remark}
This height is in fact the \emph{stable} Faltings height of the elliptic curve of $j$-invariant $\alpha$, and can also be defined using the {Hodge bundle} equipped a canonical metric \cite[Theorem 5.5.1]{Yuan-Zhang}. 
\end{remark}

The $j$-function induces an isomorphism $j \colon \mathcal{X} \cong \nobreak \mathbb{P}^1_\Z $ over $\Z $, where
\begin{itemize}
	\item the cusp point corresponds to $\infty$,
	\item the automorphic bundle $\mathcal{M}_{12}$ corresponds to $\mathcal{O}(1)$,
	\item the modular discriminant $\Delta \in
          \rH^0(\mathcal{X},\mathcal{M}_{12})$ corresponds to the
          section $x_1 \in \rH^0(\mathbb{P}^1_{\Z},\mathcal{O}(1))$. 
\end{itemize}

Let $\|\cdot\|_{\hyp}$ be the metric on
$\mathcal{O}(1)(\C )$ corresponding to the Petersson metric on
$M_{12}(\C )$ and let $\ginf_{\hyp} \colon \C
\longrightarrow \R$ be the induced Green function. By the previous
identifications it is given by
\begin{displaymath}
  \ginf_{\hyp} (z)=-\log\|\Delta (\tau _z)\|_{\Pet},
\end{displaymath}
where $\tau _{z}\in \HH$ is sent to $z$ by $j$.
Note that
we have the asymptotic estimate \cite[Section 3.2]{BurgosGil_essMinFaltings} 
\[
\ginf_{\hyp}(z) = \log|z| - 6\log\log|z| + \nobreak O(1), \quad
\text{as $|z| \longrightarrow \infty$}, 
\]
so $\ginf_{\hyp}$ is a Green function in the sense of Definition
\ref{log singular}. Thus Theorem
\ref{strong_duality_essential_minimum} and  Corollary \ref{cor:1} apply and we get 
\begin{theorem} \label{Faltings_height_1} 
	\[
	\sup_{\substack{n\geq 1\\ s \in
            \rH^0(\mathbb{P}^1_{\Z},\mathcal{O}(n)) \backslash \{0\}}} -
        \frac{\log\|s\|_{\hyp,\infty}}{n} =
        12\ess(\Ht_{\mathrm{F}}) = \inf_{\mu \in
          \mathscr{P}_{\log}^{\overline{\Z}}(\C)} \int
        \ginf_{\hyp} \,\dd \mu. 
	\]
\end{theorem}

\begin{theorem} \label{Faltings_height_2}
Let $\mathscr{P}_{\log}^{\overline{\Z }}\big(
\mathcal{Y}(\C ) \big)$ be the set of probability measures that
can be approximated $\log$-weakly by Galois orbits in
$\mathcal{Y}(\overline{\Z })$. Then
	\begin{align*}
		\sup_{\substack{n\geq 1 \\ f \in \mathcal{M}_{12n}(\Gamma,\Z)\setminus\{0\}}} - \frac{\log\|f\|_{\Pet, \infty}}{n} &= \sup_{\substack{n\geq 1\\ f \in \mathcal{M}_{12n}(\Gamma,\Z)\setminus\{0\}}} - \frac{\log\|f\|_{L^2}}{n} \\
		 &= 12\ess(\Ht_{\operatorname{F}})\\ &= \inf_{\mu \in \mathscr{P}_{\log}^{\overline{\Z }}\big( \mathcal{Y}(\C ) \big)} \int -\log\|\Delta\|_{\operatorname{Pet}} \,\dd\mu.
	\end{align*}
\end{theorem}

\begin{proof}
	The first equality follows from \cite[Lemma
        3.4.5]{Chinburg-Guignard-Soule}, where it is proved that the
        distorsion of sup-norm and $L^2$-norm  is subexponential. The
        rest are just a reformulation of Theorem
        \ref{Faltings_height_1}. 
\end{proof}

\begin{remark}
	In \cite{BurgosGil_essMinFaltings}, the authors used
	\[
	\sup_{\substack{n\geq 1\\ f \in
            \mathcal{M}_{12n}(\Gamma,\Z )\setminus\{0\}}}
        - \frac{\log\|f\|_{\Pet,\infty}}{n} \leq
        12\ess(\Ht_{\mathrm{F}}) \leq \inf_{\operatorname{cap}(K)=1}
        \int \ginf_{\hyp} \,\dd \mu_K
      \]
      to give numerical estimates of $\ess(\Ht_{\mathrm{F}})$. Here we
      have proved that the lower bound indeed reaches
      $\ess(\Ht_{\mathrm{F}})$, while for the upper bound to
      reach  $\ess(\Ht_{\mathrm{F}})$, we may
      need  to consider more general measures
      $\mathscr{P}_{\log}^{\overline{\Z }}(\C )$ than
      just equilibrium measures of compact sets of capacity
      one.
\end{remark}

Theorem \ref{ess_attained_by_integer} specializes to
\begin{theorem} \label{Faltings_height_4}
	The essential minimum of Faltings' height can be attained by
        a sequence of elliptic curves with good reduction everywhere. 
\end{theorem}

In fact, Theorem \ref{ess_attained_by_integer} says that we can approach the essential minimum of Faltings' height by elliptic curves with integral $j$-invariants hence having potentially good reduction everywhere \cite[Chapter VII. Proposition 5.5]{Silverman_book}. So after a finite field extension they become good reduction everywhere.

%%% Local Variables:
%%% mode: latex
%%% TeX-master: "main"
%%% End:

\section{The computability of the essential minimum}
\label{sec:comp-essent-minim}

For a particular height function, it is not obvious how to compute its essential minimum. 
{As a consequence of the strong duality Theorem \ref{strong_duality_essential_minimum}, we can construct both a decreasing sequence and an increasing sequence converging to the essential minimum,}
 %{\color{red} As a matter
%applying Balla\"y's theorem we can construct an increasing sequence of real numbers converging to the essential minimum.  As we will see in
%this section, as a consequence of the strong duality property we can
%also construct a decreasing sequence converging to the essential
%minimum,} 
thus showing that this invariant is a computable real number.
The obtained algorithm is far from being practical, but might be a
first step in the search of an effective procedure for this problem.

To make these ideas precise we will use the theory of computability.
We refer the reader to \cite{PI:cap} for the preliminaries on this
theory.

\subsection{Computability}
\label{sec:computable-numbers}
We first recall the notion of computable number in one of its
equivalent definitions.

\begin{definition}
  \label{def:2}
  A real number $r$ is \emph{computable} if there exists a Turing
  machine  that given any rational number $\varepsilon >0$ produces  a rational number $q$ with
  $|r-q|<\varepsilon $. A complex number is \emph{computable} if both its
  real and imaginary part are computable.
\end{definition}

The set of computable numbers is a countable subfield of $\C$ that
contains all algebraic numbers, and as a consequence most real numbers
are non-computable.  Nevertheless it is very difficult to give a
concrete non-computable number. Examples of such numbers are Chaitin's
constants, which are associated to Turing's halting problem.

There are weaker notions of computability.

\begin{definition}
  \label{def:3}
  A real number $r$ is \emph{left computable} if there is a Turing machine
  that given a natural number $n$ produces a rational number $x_n$ such that
  $\sup_{n}x_n=r$.  Similarly $r$ is \emph{right computable} if there
  is a Turing machine that given $n$ produces a rational number $y_n$ such that
  $\inf_{n} x_n=r$.
\end{definition}

\begin{lemma}
  A real number is computable if and only if it is both left and right computable.
\end{lemma}

\begin{proof}
  Assume that $r$ is computable. Then there is a Turing machine that
  given $n$ produces a rational number $z_n$ with $|r-z_{n}|<1/n$. The
  sequences $x_n=z_n-1/n$ and $y_n=z_n+1/n$ show that $ r$ is both
  left and right computable.

  Now assume that $r$ is both right and left computable and
  $0<\varepsilon\in \Q $. Then there is a Turing machine that given a
  natural number $N$ produces $\alpha _{N}=\max_{n\le N} x_n$ and
  $\beta _{N}=\min _{n\le N}y_{n}$. When
  $\beta _{N}-\alpha _{N}<\varepsilon $ we have reached the desired
  precision and the algorithm stops.
\end{proof}

\begin{definition}
  \label{def:4}
  A function $f\colon \N\longrightarrow \N$ is \emph{computable} if
  there is a Turing machine that given input $n$ produces the output
  $f(n)$.
\end{definition}

% A \emph{rational rectangle} in the complex plane is a subset
% \[
% E \coloneqq \left\{ z \mid a_1 \leq \Real(z) \leq b_1, a_2 \leq \Img(z) \leq b_2 \right\}
% \] where $a_1<b_1, a_2<b_2$ are rational numbers.

% \begin{definition}
%     Let $E \subseteq \C$ be a rational rectangle. A function
%   $\varphi\colon E \longrightarrow \mathbb{R}$ is \emph{computable} if it
%   satisfies the following conditions:
%   \begin{enumerate}
%   \item there is a Turing machine
%     such that given $z \in E \cap \mathbb{Q}^2 $ and $0<\varepsilon \in \Q$, 
%     produces an element $t \in \Q$ with $|\varphi(z)-t|<\varepsilon $,
%   \item there is a computable function $f\colon \N\longrightarrow \N$ such that
%     for all $0<n\in \N$ and all $w,z\in E$ with $ |z-w|< 1/f(n)$
%     we have that $|\varphi(z)-\varphi(w)|<1/n$.
%   \end{enumerate}
% \end{definition}

% {\color{blue}
%   \begin{remark}
%     \label{rem:1}
%     This is equivalent to \cite[Section 0.3,
%     Definition A]{PI:cap} via computable sequences. The equivalence follows directly from the fact that there exists a computable enumeration of all rational numbers.
% \end{remark}
% }

% {\color{red} If we can do this then the computation of
%     the definite integral of a computable function is justified by
%     \cite[Section 0.5, Theorem 5]{PI:cap} and that of the infimum by
%     \cite[Section 0.6, Theorem~7]{PI:cap}. }

The Green functions that will be well suited for our computations are
those which are computable in appropriate compact sets and have  effective
asymptotics.

\begin{definition} A Green function
  $\ginf \colon \C\longrightarrow \R$ is \emph{computable} if it
  satisfies the conditions:
  \begin{enumerate}
  \item for any rational rectangle $E$ {(i.e. $E = \{ z \in \mathbb{C} \colon a_1 \leq \operatorname{Re}(z) \leq b_1, a_2 \leq \operatorname{Im}(z) \leq b_2 \}$ for rational numbers $a_1<b_1, a_2<b_2$)} the restriction $\ginf|_{E} \colon
    E \longrightarrow \R$ is computable in the sense of \cite[Section
    0.3, Definition A]{PI:cap};
  \item there exists a computable function
    $f \colon \N\longrightarrow \N$ such that
	\begin{displaymath}
          \big| \ginf(z)-\log|z| \big|\le\frac{1}{n} \log|z| \quad \text{  for all } z\in \C \text{ with }
          |z|>f(n).
    \end{displaymath}
  \end{enumerate}
\end{definition}

The following is our main result in this section.

\begin{theorem} \label{ess_min_computable} Let $\ginf$ be a computable
  Green function. Then $\ess(\Ht_{\ginf})$ is a computable real number.
\end{theorem}

The proof will be carried out in the next two sections, where we will
see that the essential minimum is both left and right computable.

%----------------------------------------
\subsection{Left computability}

 By Proposition~\ref{prop:5} we have
\[
\ess(\Ht_{\ginf}) = \sup_{\Lambda} \inf_{z \in \C } \Big(\ginf(z)-\sum_{i=1}^{k} a_{i}\log |Q_{i}(z)|\Big)
\]
with
\begin{displaymath}
  \Lambda = \Big\{ (Q_{1},\dots,Q_{k},a_{1},\dots,a_{k}) \, \Big|\  k \in
  \N , Q_i \in \Z[x]\setminus\{0\}, a_i \in \Q_+ \text{ and } \sum_{i=1}^k a_i
  \deg(Q_i) <1 \Big\}. 
\end{displaymath}
Since the index set $\Lambda$ can be effectively enumerated, the left
computability of $\ess(\Ht_{g})$ is a direct consequence of the next
statement.

\begin{proposition}
  Assume that $\ginf$ is a computable Green function.  Let
  $(Q_{1},\dots,Q_{k},a_{1},\dots,a_{k}) \in \Lambda$ and set
  $\varphi =\ginf-\sum_{i=1}^{k} a_{i}\log |Q_{i}|$.  Then
  $ \inf_{z \in \C } \varphi(z)$ is a computable real number.
\end{proposition}

\begin{proof}
  Since $\ginf$ is a computable Green function and
  $\sum a_i \deg(Q_i)<1$, we can determine a rational rectangle $E$ such that the
  infimum of $\varphi$ is attained in $E$.  Similarly we
  can determine $M \in \mathbb{N}$ such that there is
  $z \in E$ with $\varphi(z) \le M$ and so 
  \begin{displaymath}
    \inf_{z\in \C} \varphi(z) = \inf_{z\in E} \min \{\varphi(z),M\}.
  \end{displaymath}
  % Around the roots of $Q_i$ where $\varphi$ goes to infinity, the
  % asymptotic of $\varphi$ is effective. Thus we can also find small
  % open balls $B^o(\alpha_j,r_j)$ where $\alpha_j$ are roots of $Q_i$
  % and $r_j$ computable, such that the infimum of $\varphi$ is not
  % taken inside $B^o(\alpha_j,r_j)$.
	% Let $K=B(0,R)-\bigcup_j B^o(\alpha_j,r_j)$. After enlarging
        % $R$ and shrinking $r _{j}$ if needed, we can assume that $K$
        % is a regular closed set.  
  Then this infimum is computable because the restriction of
  $\min\{\varphi,M\}$ to $E$ is a computable function \cite[Section
  0.6, Theorem~7]{PI:cap}.
\end{proof}

%----------------------------------------
\subsection{Right computability}

For the right computability, we look at the other side of the strong
duality property. By Theorem \ref{strong_duality_essential_minimum}
and Theorem \ref{mu_PQ_construction} we have
\[
\ess(\Ht_{\ginf}) = \inf  _{(P,Q) \in \Theta} \int \ginf \,\dd\mu_{P,Q}
\]
for the index set
\begin{displaymath}
  \Theta=\big\{ (P, Q) \, \big| \ P,Q \in \Z[x] \setminus \Z \text{ with } P\ne Q \text{ monic and irreducible}\big\},
\end{displaymath}
and where for each $(P, Q) \in \Theta$ we denote by
$\mu_{P,Q} \in \mathscr{P}^{\overline{\Z}}_{\log}(\C )$ the 
measure~in~\eqref{eq:31}.

Since the index set $\Theta$ can also be effectively enumerated, the
right computability of the essential minimum follows from the next
result.

\begin{proposition}
  \label{prop:7}
  Assume that $\ginf$ is a computable Green function and let
  $(P,Q)\in \Theta$. Then the integral
  $\displaystyle \int \ginf \,\dd\mu_{P,Q} $ is computable.
\end{proposition}

\begin{proof}
Set for short $d=\deg(P)$ and $e=\deg(Q)$ and consider the function
\begin{displaymath}
\rho \colon [0,1] \longrightarrow \R, \quad 
\rho(\theta)= 	\frac{((\varphi_{P,Q})_* \ginf) (e^{2\pi i\theta})}{d\, (e+1)}.
  \end{displaymath}
  For each $\theta\in [0,1]$ consider also the polynomial
  $S_{\theta}= P^{\deg(Q)+1}-e^{2 \pi i\theta} Q^{\deg(P)} \in \C[X]$. Then
  \begin{equation}
    \label{eq:17}
\rho(\theta)= \frac{1}{d\, (e+1)} \sum_{S_{\theta}(w)=0}e_w\, \ginf (w).
  \end{equation}
  
  There exists an effective $R \in \Q_{>0}$ such that
  $\varphi_{P,Q}^{-1}(S^1) \subseteq [-R,R] \times i[-R, R]$. Since
  $\ginf$ restricted to this rectangle is computable by definition and
  the process of finding the complete set of roots of polynomials is
  computable by \cite{Specker69} (see also \cite[Proposition
  3.13]{MR4887199}) we have that $\rho$ is a computable function, and
  so the Riemann integral
  \[
	\int_0^1 (\varphi_{P,Q})_* \ginf \, (e^{2\pi i \theta}) \, \dd\theta
      \] is computable \cite[Section 0.5, Theorem 5]{PI:cap}.
\end{proof}

%%% Local Variables:
%%% mode: latex
%%% TeX-master: "main"
%%% End:

\appendix
\section{Duality in linear
  programming} \label{sec:dual-line-optim}

In this appendix we present a coordinate-free formulation of duality in linear
optimization and recall the proof of the strong duality property in
the finite dimensional case.  This allows to place the optimization
problems from Section \ref{section_strong_duality} within a general
framework and to give a more conceptual approach to their duality
properties.

Let $E,F$ be two real vector spaces equipped with a pairing
$\langle -,- \rangle \colon E\times F\longrightarrow \R $. Let
$E\longrightarrow F^{\vee}$ and $F\longrightarrow E^{\vee}$ be the induced linear maps between
these spaces and their duals, that we respectively denote by
$x\longmapsto x^{\dag}$ and $y\longmapsto y^{\dag}$. They are defined by
setting
\begin{displaymath}
x^{\dag}(y)=y^{\dag}(x) = \langle x,y\rangle \quad \text{ for all } x\in E \text{ and } y\in F.
\end{displaymath}
Consider two convex cones $\sigma\subset E$ and $\tau\subset F$. Their duals are the  convex cones 
\begin{displaymath}
  \sigma^{\vee}= \{ u\in E^{\vee} \mid u(x)\ge 0 \text{ for all } x\in \sigma\}, \quad \tau^{\vee}= \{v\in F^{\vee} \mid v(y)\ge 0 \text{ for all } y\in \tau\}.
\end{displaymath}
Let  also $u_{0}\in E^{\vee}$ and $v_{0}\in F^\vee$.

 \begin{definition}
   \label{def:1}
   The \emph{primal problem} and the \emph{dual problem} for this
   datum are the optimization problems respectively given by
   \begin{displaymath}
      \primal=\inf\{ u_{0}(x) \mid   x\in \sigma, \ x^{\dag} -v_{0} \in \tau^{\vee}\} , \quad \dual=\sup\{ v_{0}(y)  \mid y\in \tau , \ u_{0}-y^{\dag} \in \sigma^{\vee}\}.
     \end{displaymath}
     We refer to the quantities $\mathcal{P}$ and $\dual$ as the \emph{optimal
       values} of these problems.
\end{definition}

\begin{remark}
  \label{rem:2}
  The role of the primal and the dual problems can be exchanged:  denote by $\mathcal{P}^{\op}$ and
  $\dual^{\op}$ the primal and dual problems that arise when
  swapping the vector spaces $E$ and $F$ and considering the cones
  $-\tau\subset F$ and $ -\sigma\subset E$ together with the functionals
  $-v_{0}\in F^{\vee}$ and $-u_{0}\in E^{\vee}$. Then it can be easily
  verified that $\mathcal{P}=-\dual^{\op}$ and
  $\dual=-\mathcal{P}^{\op}$.
\end{remark}

The classical problems in linear programming are a particular case of
this framework.

  \begin{example}
\label{exm:2}
Set $E=\R ^{m}$ and
  $F= \R ^{n}$ and given $A\in \R ^{m\times n}$,
  $b\in \R ^{m}$ and $c\in \R ^{n}$ consider the pairing
  and functionals defined by
  \begin{displaymath}
(x,y)\longmapsto\langle x, y \rangle= x^{\transpose}A \, y, \quad x\longmapsto u_{0}(x)= b^{\transpose}x, \quad  y\longmapsto v_{0}(y)=c^{\transpose}y
\end{displaymath}
together with the cones $\sigma=\R ^{m}_{\ge 0}$ and
$\tau=\R ^{n}_{\ge 0}$. Then the associated primal and dual problems
boil down to the usual forms
  \begin{displaymath}
   \primal= \inf\{ b^{\transpose} x  \mid x \ge 0, \ A^{\transpose}x \ge c\} 
    ,\quad 
    \dual= \sup\{c^{\transpose} y \mid y\ge 0, \ A\, y \le b\},
  \end{displaymath}
where  $\le$ and $\ge $ on these real vector spaces  means
  that these inequalities hold coordinate-wise.
\end{example}

The problems from Section \ref{section_strong_duality} also fit within
this framework.  Denote by $\csM_{\log}(\C)$ and
$\mathscr{S}_{\log}(\C)$ the cone and the vector space generated by
$\csP_{\log}(\C)$. The elements of $\csM_{\log}(\C)$ are the measures
on $\C $ that integrate the function $\log^{+}|z|$, whereas
those of $\mathscr{S}_{\log}(\C)$ are the differences of these
measures.
 
We also let $\csP'_{\log}(\C)$ be the set of probability measures
$\mu\in \csP_{\log}(\C)$ that integrate the functions $\log |Q|$ for
all $Q\in \Z[x]$, and we respectively denote by $ \csM'_{\log}(\C)$
and $ \mathscr{S}'_{\log}(\C )$ the cone and vector space
generated by this set of probability measures. 

\begin{example} 
  \label{exm:1}
  Set 
  \begin{displaymath}
    E=\mathscr{S}'_{\log}(\C )  \quad\text{ and }\quad F= \R \oplus \bigoplus_{n\in \mathbb{N}}\R .
  \end{displaymath}
  Fix an enumeration $Q_{1}, Q_{2}, Q_{3}, \dots$ of all nonconstant
  integer polynomials and recall that the elements of $F$ are the
  tuples $ a=(a_{0},a_{1},a_{2},\dots) $ with $ a_{n} =0$ for all but
  a finite number of~$n$'s. We  then consider the pairing
  $E\times F\longrightarrow \R $ defined by
  \begin{displaymath}
    (\mu,a)\longmapsto a_{0}\int \dd\mu +\sum_{n\in \mathbb{N}} a_{n}
    \int  \log|Q_{n}| \, \dd \mu.
  \end{displaymath}
  We also consider the convex cones defined as
  \begin{equation*}
    \sigma =\mathscr{M}_{\log}'(\C ) \subset E \quad \text{ and }
    \quad 
    \tau =\{(a_0,a_1,\dots,a_n,\dots)\mid a_n\ge 0\text{ for all
           }n>0\}\subset F.  
  \end{equation*}
  Let $\ginf \colon \C  \longrightarrow \R $ be a
  continuous function that is asymptotically logarithmic at $\infty$ in the sense of
  Definition~\ref{log singular}, and define the functionals
  $u_{0}\in E^{\vee}$ and $v_{0}\in F^{\vee}$ as
  \begin{displaymath}
    u_{0}(\mu)= \int \ginf \, \dd \mu  \quad\text{ and }\quad  v_{0}(a)= a_{0}.
  \end{displaymath}
  % We then denote by $\mathcal{P}(g_{\infty})$ and
  % $\dual(g_{\infty})$ the optimal values of the primal and dual
  % problems corresponding to this datum.

  The associated  primal problem amounts to the minimization
  \begin{math}
   \primal(\ginf) = \inf_{\mu}\int \ginf\, \dd\mu
  \end{math}
  over the measures $\mu \in \mathscr{M}_{\log}'(\C )$ such
  that
  \begin{displaymath}
    a_{0}\left(\int \dd\mu -1\right)  + \sum_{n\in \mathbb{N}} a_{n}\int
    \log|Q_{n}| \, \dd\mu \ge 0 
  \end{displaymath}
  for all $a_{0}\in \R $ and $ a_{n}\in \R _{\ge 0}$,
  $n\in \mathbb{N}$, with $a_{n}=0$ for all but a finite number of
  $n$'s. Since $a_0$ is arbitrary, this forces $\mu \in
  \csP_{\log}'(\C)$. Hence, this minimization is over the probability
  measures $\mu \in \csP_{\log}'(\C )$ such that
  $\int \log|Q_{n}| \, \dd \mu \ge 0 $ for all $n$, which coincide with
  those in $ \csP_{\log}(\C )$ satisfying the same
  condition. Hence
  \begin{displaymath}
   \primal(\ginf) = \inf \Big\{ \int \ginf\, \dd\mu
    \, \Big| \ \mu \in \mathscr{P}_{\log}(\C ), \, \int \log|Q_{n}| \, \dd \mu \ge 0 \text{ for all } n\in \mathbb{N}\Big\}
  \end{displaymath}
  as in \eqref{eq:51}. Similarly, the associated dual problem is the
  maximization $ \dual(\ginf) = \sup_{a} a_{0} $ over
  $a\in \tau $ such that
  \begin{equation}
    \label{eq:1}
    \int \ginf \, \dd\mu - a_{0} \int \dd \mu - \sum_{n\in
      \mathbb{N}} a_{n}\int \log|Q_{n}| \, \dd\mu \ge 0 \quad \text{
      for all }
    \mu\in \mathscr{M}_{\log}'(\C ).
  \end{equation}
  This is equivalent to the inequality
  \begin{math}
    \ginf(z) -\sum_{n\in \mathbb{N}} a_{n}\log|Q_{n}(z)| \ge a_{0}
  \end{math}
  for all $z\in \C $, as it can be seen by considering
  (\ref{eq:1}) for the Dirac delta measures $\mu=\delta_{z}$ for all
  $z\in \C  \setminus \overline{\mathbb{Q}}$.  Hence
  \begin{displaymath}
    \dual(\ginf) = \sup_{a\in \tau}   \inf_{x\in \C } \Big( \ginf(x) - \sum_{n\in \mathbb{N}} a_{n}\log |Q_{n}(x)|\Big),
  \end{displaymath}
  in agreement with \eqref{eq:6}.
\end{example}

The weak duality property is the fact that the optimal value of the
primal problem bounds above that of the dual, and follows readily from
the definitions.

\begin{proposition}
  \label{prop:2} We have
  $\mathcal{P}\ge \dual$.
\end{proposition}

\begin{proof}
    For all $x\in \sigma $ and $y\in \tau $ such that $x^{\dag} -v_{0}
  \in \tau^{\vee}$ and $u_{0}-y^{\dag} \in \sigma^{\vee}$ we have
  $u_{0}(x) \ge \langle x,y\rangle \ge v_{0}(y)$, which gives the
  inequality.
\end{proof}

The strong duality property is the equality between the optimal values
of the primal and the dual problem.  We end this appendix by recalling
the proof of this property in the finite dimensional situation. 

\begin{definition}
  The \emph{feasibility set} of the primal problem $\mathcal{P}$
  and the dual problem $\dual$ are
\begin{displaymath}
 S_{\mathcal{P}}=       \{ x\in \sigma \mid x^{\dag} -v_{0} \in \tau^{\vee}\} \subset E
  \quad \text{ and } \quad S_{\dual}= \{ y\in \tau  \mid u_{0}-y^{\dag} \in \sigma^{\vee}\} \subset F.
\end{displaymath}
We say that $\mathcal{P}$ (respectively $\dual$) is
\emph{feasible} if $S_{\mathcal{P}}\not = \emptyset$ (respectively if
$S_{\dual}\ne \emptyset$), and we  say that $\mathcal{P}$
(respectively $\dual$) is \emph{bounded} if the set
$ \{u_{0}(x)\colon x\in S_{\mathcal{P}}\} $ is bounded below
(respectively if the set $\{ v_{0}(y)\colon y\in S_{\dual}\}$ is
bounded above).

We also say that $\mathcal{P}$ (respectively $\dual$) is
\emph{attained} if there exists $x\in S_{\mathcal{P}}$ such that
$\mathcal{P}=u_{0}(x)$ (respectively if there exists
$y\in S_{\dual}$ such that $\dual=v_{0}(y)$).
\end{definition}

The primal problem is feasible and bounded if and only if $\cP\in \R$,
and similarly for the dual problem.  Moreover, the weak duality
property shows that if one of these problems is feasible then the
other is bounded.

\begin{theorem}\label{thm:1} Assume that $E$ and $F$ are finite
  dimensional vector spaces and that $\sigma $ and $\tau $ are closed convex
  cones. The following conditions are equivalent:
  \begin{enumerate}
  \item \label{item:3} the primal problem $\mathcal{P}$ is feasible
    and bounded;
  \item \label{item:4} the dual problem  $\dual$ is  feasible and bounded.
  \end{enumerate}
  If any of these conditions holds then $\cP=\cD \in \R $ and
  both $\mathcal{P}$ and $\dual$ are attained.
\end{theorem}

\begin{proof}
  First assume that \eqref{item:3} holds and consider the closed
  convex subsets of $\R \times F^{\vee}$ defined~as
  \begin{displaymath}
    V=\{(u_{0}(x), x^{\dag})  \, | \   x\in \sigma \} \quad \text{ and }
    \quad W_{\lambda}=\{(t,v+v_{0}) \mid t\le \lambda, \ v\in \tau^{\vee}\}
    \quad \text{ for } \lambda\in \R .
  \end{displaymath}
  For each $\lambda\in \R $ we have that
  $V\cap W_{\lambda}\ne \emptyset$ if and only if there exists
  $x\in S_{\mathcal{P}}$ such that $u_{0}(x) \le \lambda$.

  We have the decomposition
  $W_{\lambda}=\R _{\le 0}\times \tau^{\vee}+(p,v_{0})$. Hence
  considering the closed convex cone   and the point respectively defined as
  \begin{displaymath}
    C=V-\R _{\le 0}\times \tau=
    \{(u_{0}(x)-t, x^{\dag}-v) : t\in \R _{\le 0}, \ x\in \sigma, \ v\in \tau^{\vee}\} \quad \text{ and } \quad
    p_{\lambda}=(\lambda,v_{0}),
  \end{displaymath}
  the condition $V\cap W_{\lambda}\ne \emptyset$ turns out to be
  equivalent to $p_{\lambda}\in C$.  Since $C$ is a closed cone, this
  condition on $\lambda\in \R $ is closed, and it is also
  nonempty and bounded below because $\mathcal{P}$ is feasible and
  bounded.  Thus setting
  $ \lambda_{0}=\inf\{\lambda \, | \ p_{\lambda}\in C\}$ we have
 \begin{displaymath}
  \primal=\lambda_{0} =\min\{\lambda \, |\  p_{\lambda}\in C\}\in \R .
  \end{displaymath}
  In particular $\mathcal{P}$ is attained.

  Now let $\lambda<\lambda_{0}$. By the point-cone separation theorem
  there exists
  $h\in (\R \times F^{\vee})^{\vee}= \R \oplus F$ such
  that $h|_{C}\ge 0$ and $h(p_{\lambda})<0$, which implies that
  $h|_{V} >h|_{W_{\lambda}}$. Hence
  writing $h=(b,-y)$ with $b\in \R $ and $y\in F$ we
  have
  \begin{equation}
    \label{eq:2}
 b\, u_{0}(x)-\langle x,y\rangle >    b\, t-v(y)-v_{0}(y) \quad \text{   for all }
    t\le \lambda, x\in \sigma, v\in \tau^{\vee}.
  \end{equation}
  Specializing  \eqref{eq:2} to
 $x=x_{1}\in S_{\mathcal{P}}\subset \sigma$ and 
  $v=v_{1}=x_{1}^{\dag}-v_{0}\in \tau^{\vee}$  we get
  $ b\, u_{0}(x_{1}) > b\, t$, which implies that $b>0$ because $t$
  can be an arbitrarily large negative number.

  We   assume without loss of generality that $b=1$. Then 
  \eqref{eq:2} specialized to $t=\lambda$  becomes 
  \begin{equation}
\label{eq:7}
u_{0}(x)-\langle x,y\rangle  >     \lambda-v(y)-v_{0}(y)   
     \quad \text{   for all }
    x\in \sigma, v\in \tau^{\vee}.
  \end{equation}
  Specializing this inequality to $x=0$ gives $v(y)\ge 0$ for all
  $v\in \tau^{\vee}$, whereas taking instead $v=0$ gives
  $(u_{0}-y^{\dag})(x)\ge 0$ for all $x\in \sigma$. Hence
  $u_{0}-y^{\dag}\in \sigma^{\vee}$, and since $\tau $ is assumed to
  be closed we also have that $y\in \tau$, and so the dual problem
  $\dual$ is feasible. We also have that $\dual$ is bounded because
  $\mathcal{P}$ is feasible, thus proving the
  condition~\eqref{item:4}. Moreover~\eqref{eq:7} specialized to $x=0$
  and $v=0$ gives
  \begin{displaymath}
\dual \ge      v_{0}(y)> \lambda, 
  \end{displaymath}
  and since $\lambda$ can be arbitrarily close to $\mathcal{P}$ we
  obtain $\mathcal{P}\le \dual$. Combining with the weak duality
  property (Proposition \ref{prop:2}) we conclude that
  $\mathcal{P}= \dual$, as stated.
    
  Finally, the case when \eqref{item:4} holds we reduce to the
  previous situation using Remark \ref{rem:2}.
 \end{proof}

 \begin{remark}
   \label{rem:4}
   The proof of Theorem \ref{strong_duality} follows this approach,
   with suitable modifications due to the fact that there  we  deal
   with some specific infinite dimensional vectors spaces and cones
   that are not necessarily closed.
 \end{remark}

%%% Local Variables:
%%% mode: latex
%%% TeX-PDF-mode: t
%%% TeX-source-correlate-mode: t
%%% TeX-master: "main"
%%% End: 

%----------------------------------------
\bibliography{ref.bib}

@book{PI:cap,
	author = {{Pour-El}, Marian Boykan and Richards, J. Ian},
	isbn = {3-540-50035-9},
	mrclass = {03F60 (03D80 46-02 46R05 47-02 65Jxx)},
	mrnumber = {1005942},
	mrreviewer = {Rodney\ G.\ Downey},
	pages = {xii+206},
	publisher = {Springer-Verlag, Berlin},
	series = {Perspectives in Mathematical Logic},
	title = {Computability in analysis and physics},
	year = {1989}}

@article {MR4887199,
AUTHOR = {Barbieri, Sebasti\'an and {Carrasco-Vargas}, Nicanor and Rojas,
Crist\'obal},
TITLE = {Effective dynamical systems beyond dimension zero and factors
of {SFT}s},
JOURNAL = {Ergodic Theory Dynam. Systems},
FJOURNAL = {Ergodic Theory and Dynamical Systems},
VOLUME = {45},
YEAR = {2025},
NUMBER = {5},
PAGES = {1329--1369},
ISSN = {0143-3857,1469-4417},
}

@article{BostGilletSoule,
	author = {Jean-Beno{\^\i}t Bost and Gillet, Henri and Soul\'e, Christophe},
	fjournal = {Journal of the American Mathematical Society},
	issn = {0894-0347,1088-6834},
	journal = {J. Amer. Math. Soc.},
	mrclass = {14G40 (11G35 14C17)},
	mrnumber = {1260106},
	mrreviewer = {Yoichi\ Miyaoka},
	number = {4},
	pages = {903--1027},
	title = {Heights of projective varieties and positive {G}reen forms},
	volume = {7},
	year = {1994}}

@incollection{Specker69,
	author = {Specker, Ernst},
	booktitle = {Constructive {A}spects of the {F}undamental {T}heorem of {A}lgebra ({P}roc. {S}ympos., {Z}\"urich-{R}\"uschlikon, 1967)},
	pages = {321--329},
	publisher = {Wiley-Interscience [A Division of John Wiley \& Sons, Ltd.], London-New York-Sydney},
	title = {The fundamental theorem of algebra in recursive analysis},
	year = {1969}}

@book{Silverman_book,
	author = {Joseph Hillel Silverman},
	edition = {Second},
	isbn = {978-0-387-09493-9},
	mrclass = {11-02 (11G05 11G20 14H52 14K15)},
	mrnumber = {2514094},
	mrreviewer = {Vasil\cprime\ \=I.\ Andr\=\i\u ichuk},
	pages = {xx+513},
	publisher = {Springer, Dordrecht},
	series = {Graduate Texts in Mathematics},
	title = {The arithmetic of elliptic curves},
	volume = {106},
	year = {2009}}

@article{Faltings_Finiteness,
	author = {Faltings, Gerd},
	fjournal = {Inventiones Mathematicae},
	journal = {Invent. Math.},
	number = {3},
	pages = {349--366},
	title = {Endlichkeitss\"atze f\"ur abelsche {V}ariet\"aten \"uber {Z}ahlk\"orpern},
	volume = {73},
	year = {1983}}

@book{Burgos_book_2014,
	author = {{Burgos Gil}, Jos\'{e} Ignacio and Philippon, Patrice and Sombra, Mart\'{\i}n},
	series = {Ast\'{e}risque},
	isbn = {978-2-85629-783-4},
	issn = {0303-1179,2492-5926},
	publisher = {Société Mathématique de France},
	mrclass = {14G40 (14M25 32P05 52A41)},
	mrnumber = {3222615},
	mrreviewer = {Fabien\ Pazuki},
	pages = {vi+222},
	title = {Arithmetic geometry of toric varieties. {M}etrics, measures and heights},
	volume = {360},
	year = {2014}}

@book{Yuan-Zhang,
	author = {Yuan, Xinyi and Zhang, Shou-Wu},
	isbn = {9780691278704},
	lastchecked = {2026-01-13},
	publisher = {Princeton University Press},
	series = {Annals of Mathematics Studies},
	title = {Adelic line bundles on quasi-projective varieties},
	year = {2026}}

@misc{szachniewicz2023,
	archiveprefix = {arXiv},
	author = {Micha{\l} Szachniewicz},
	eprint = {2306.06275},
	primaryclass = {math.LO},
	title = {Existential closedness of {$\overline{\mathbb{Q}}$} as a globally valued field via {A}rakelov geometry},
	year = {2023}}

@article{Smyth:mtraiII,
	author = {Christopher James Smyth},
	fjournal = {Mathematics of Computation},
	issn = {0025-5718,1088-6842},
	journal = {Math. Comp.},
	mrclass = {12A15 (10F25)},
	mrnumber = {616373},
	mrreviewer = {Marthe\ Grandet},
	number = {155},
	pages = {205--208},
	title = {On the measure of totally real algebraic integers. {II}},
	volume = {37},
	year = {1981}}

@article{Smyth:tpaist,
	author = {Christopher James Smyth},
	journal = {Ann. Inst. Fourier (Grenoble)},
	number = {3},
	pages = {1--28},
	title = {Totally positive algebraic integers of small trace},
	volume = {34},
	year = {1984}}

@book{Billingsley,
	author = {Billingsley, Patrick},
	edition = {Second},
	mrclass = {60B10 (28A33 60F17)},
	mrnumber = {1700749},
	pages = {x+277},
	publisher = {John Wiley \& Sons, Inc., New York},
	series = {Wiley Series in Probability and Statistics: Probability and Statistics},
	title = {Convergence of probability measures},
	year = {1999}}

@article{Bi97,
	author = {Bilu, Yuri},
	fjournal = {Duke Mathematical Journal},
	issn = {0012-7094},
	journal = {Duke Math. J.},
	mrclass = {11G35 (11G25 11J68 14G05 14G25)},
	mrnumber = {1470340 (98m:11067)},
	mrreviewer = {Dan Abramovich},
	number = {3},
	pages = {465--476},
	title = {Limit distribution of small points on algebraic tori},
	volume = {89},
	year = {1997}}

@article{BLW,
	author = {Bloom, Thomas and Levenberg, Norman and Wielonsky, Frank},
	fjournal = {Computational Methods and Function Theory},
	issn = {1617-9447,2195-3724},
	journal = {Comput. Methods Funct. Theory},
	mrclass = {60F10 (31B15)},
	mrnumber = {3428818},
	number = {4},
	pages = {555--594},
	title = {Logarithmic potential theory and large deviation},
	volume = {15},
	year = {2015}}

@book{BombieriGubler,
	author = {Bombieri, Enrico and Gubler, Walter},
	isbn = {978-0-521-84615-8; 0-521-84615-3},
	mrclass = {11G50 (11-02 11G10 11G30 11J68 14G40)},
	mrnumber = {2216774},
	mrreviewer = {Yuri\ Bilu},
	pages = {xvi+652},
	publisher = {Cambridge University Press, Cambridge},
	series = {New Mathematical Monographs},
	title = {Heights in {D}iophantine geometry},
	volume = {4},
	year = {2006}}

@article{Lobrich,
	author = {L{\"o}brich, Steffen},
	fjournal = {Journal de Th\'eorie des Nombres de Bordeaux},
	issn = {1246-7405,2118-8572},
	journal = {J. Th\'eor. Nombres Bordeaux},
	mrclass = {11G50 (14G40 14H52)},
	mrnumber = {3614528},
	mrreviewer = {Shu\ Kawaguchi},
	number = {1},
	pages = {289--305},
	title = {A gap in the spectrum of the {F}altings height},
	volume = {29},
	year = {2017}}

@article{OSLW,
	author = {{Orive Rodr{\'\i}guez}, Ram{\'o}n Angel and {S{\'a}nchez Lara}, Joaqu{\'\i}n Francisco and Wielonsky, Franck},
	fjournal = {Journal of Approximation Theory},
	issn = {0021-9045,1096-0430},
	journal = {J. Approx. Theory},
	mrclass = {31C15 (47J40)},
	mrnumber = {3935952},
	mrreviewer = {B\'ela\ Nagy},
	pages = {71--100},
	title = {Equilibrium problems in weakly admissible external fields created by pointwise charges},
	volume = {244},
	year = {2019}}

@incollection{Rumely,
	author = {Rumely, Robert},
	booktitle = {Spectral problems in geometry and arithmetic ({I}owa {C}ity, {IA}, 1997)},
	isbn = {0-8218-0940-7},
	mrclass = {11G50 (14G40)},
	mrnumber = {1710794},
	mrreviewer = {Antoine\ Chambert-Loir},
	pages = {159--166},
	publisher = {Amer. Math. Soc., Providence, RI},
	series = {Contemp. Math.},
	title = {On {B}ilu's equidistribution theorem},
	volume = {237},
	year = {1999}}

@article{Serre19,
	author = {Serre, Jean-Pierre},
	fjournal = {Ast\'erisque},
	isbn = {978-2-85629-915-9},
	issn = {0303-1179,2492-5926},
	journal = {Ast\'erisque},
	mrclass = {11K36 (11B05 11G10 11R06)},
	mrnumber = {4093205},
	mrreviewer = {Ben\ Joseph\ Green},
	note = {S\'eminaire Bourbaki. Vol. 2017/2018. Expos\'es 1136--1150},
	pages = {Exp. No. 1146, 379--425},
	title = {Distribution asymptotique des valeurs propres des endomorphismes de {F}robenius [d'apr\`es {A}bel, {C}hebyshev, {R}obinson,{$\ldots$}]},
	volume = {414},
	year = {2019}}

@misc{BallaySombra,
	archiveprefix = {arXiv},
	author = {Balla\"{y}, Fran\c{c}ois and Mart\'in Sombra},
	eprint = {2407.14978},
	primaryclass = {math.NT},
	title = {Approximation of adelic divisors and equidistribution of small points},
	year = {2025}}

@article{BPSminima,
	author = {{Burgos Gil}, Jos\'e{} Ignacio and Philippon, Patrice and Sombra, Mart\'in},
	fjournal = {Universit\'e{} de Grenoble. Annales de l'Institut Fourier},
	issn = {0373-0956,1777-5310},
	journal = {Ann. Inst. Fourier (Grenoble)},
	mrclass = {14G40 (11G35 11G50 14M25)},
	mrnumber = {3449209},
	mrreviewer = {Sho\ Tanimoto},
	number = {5},
	pages = {2145--2197},
	title = {Successive minima of toric height functions},
	volume = {65},
	year = {2015}}

@article{BPRStoric,
	author = {{Burgos Gil}, Jos\'e{} Ignacio and Philippon, Patrice and {Rivera-Letelier}, Juan and Sombra, Mart\'in},
	fjournal = {American Journal of Mathematics},
	issn = {0002-9327,1080-6377},
	journal = {Amer. J. Math.},
	mrclass = {11G50 (11G35 14G40 14M25)},
	mrnumber = {3928039},
	mrreviewer = {J\"org\ Jahnel},
	number = {2},
	pages = {309--381},
	title = {The distribution of {G}alois orbits of points of small height in toric varieties},
	volume = {141},
	year = {2019}}

@article{BurgosGil_essMinFaltings,
	author = {{Burgos Gil}, Jos\'{e} Ignacio and Menares, Ricardo and {Rivera-Letelier}, Juan},
	fjournal = {Mathematics of Computation},
	issn = {0025-5718,1088-6842},
	journal = {Math. Comp.},
	mrclass = {11F11 (11G50 14G40 37P30)},
	mrnumber = {3802441},
	mrreviewer = {Matthew\ A.\ Papanikolas},
	number = {313},
	pages = {2425--2459},
	title = {On the essential minimum of {F}altings' height},
	volume = {87},
	year = {2018}}

@article{Chinburg-Guignard-Soule,
	author = {Ted Chinburg and Quentin Guignard and Christophe Soul{\'e}},
	issn = {0022-314X},
	journal = {J. Number Theory},
	pages = {294-341},
	title = {On the slopes of the lattice of sections of hermitian line bundles},
	volume = {228},
	year = {2021}}

@article{Kuhne,
	author = {K{\"u}hne, Lars},
	fjournal = {Journal of the European Mathematical Society (JEMS)},
	issn = {1435-9855,1435-9863},
	journal = {J. Eur. Math. Soc. (JEMS)},
	mrclass = {14G40 (11G10 11G50 14G05 14K15)},
	mrnumber = {4404794},
	mrreviewer = {\'Eric\ Gaudron},
	number = {6},
	pages = {2077--2131},
	title = {Points of small height on semiabelian varieties},
	volume = {24},
	year = {2022}}

@article{Autissier,
	author = {Pascal Autissier},
	issn = {1631-073X},
	journal = {C. R. Math.},
	number = {9},
	pages = {639-641},
	title = {Sur une question d'{\'e}quir{\'e}partition de nombres alg{\'e}briques},
	volume = {342},
	year = {2006}}

@article{Smith-Orloski-Sardari,
	author = {Bryce Joseph Orloski and Naser Talebizadeh Sardari and Alexander Smith},
	journal = {Math. Comput.},
	number = {354},
	pages = {2005--2040},
	title = {New lower bounds for the Schur-Siegel-Smyth trace problem},
	volume = {94},
	year = {2024}}

@article{YuanBig,
	author = {Yuan, Xinyi},
	fjournal = {Inventiones Mathematicae},
	issn = {0020-9910,1432-1297},
	journal = {Invent. Math.},
	mrclass = {14G40 (11G50 14C20)},
	mrnumber = {2425137},
	mrreviewer = {Antoine\ Chambert-Loir},
	number = {3},
	pages = {603--649},
	title = {Big line bundles over arithmetic varieties},
	volume = {173},
	year = {2008}}

@article{Qu-Yin,
	author = {Qu, Binggang and Yin, Hang},
	fjournal = {Advances in Mathematics},
	issn = {0001-8708,1090-2082},
	journal = {Adv. Math.},
	mrclass = {14G40 (11G50)},
	mrnumber = {4803242},
	pages = {Paper No. 109961, 24},
	title = {Arithmetic {D}emailly approximation theorem},
	volume = {458},
	year = {2024}}

@article{Ballay,
	author = {Balla\"{y}, Fran\c{c}ois},
	fjournal = {Compositio Mathematica},
	issn = {0010-437X,1570-5846},
	journal = {Compos. Math.},
	mrclass = {14G40 (11G50 11H06)},
	mrnumber = {4271920},
	mrreviewer = {Ariyan\ Javanpeykar},
	number = {6},
	pages = {1302--1339},
	title = {Successive minima and asymptotic slopes in {A}rakelov geometry},
	volume = {157},
	year = {2021}}

@book{Pollard,
	author = {Pollard, David},
	isbn = {0-387-90990-7},
	mrclass = {60F05 (60B10)},
	mrnumber = {762984},
	mrreviewer = {R.\ M.\ Dudley},
	pages = {xiv+215},
	publisher = {Springer-Verlag, New York},
	series = {Springer Series in Statistics},
	title = {Convergence of stochastic processes},
	year = {1984}}

@book{Hormander,
	author = {H{\"o}rmander, Lars},
	isbn = {978-0-8176-4584-7; 0-8176-4584-5},
	mrclass = {26B25 (31C10 32Txx 35A27 52A40)},
	mrnumber = {2311920},
	note = {Reprint of the 1994 edition.},
	pages = {viii+414},
	publisher = {Birkh\"auser Boston, Inc., Boston, MA},
	series = {Modern Birkh\"auser Classics},
	title = {Notions of convexity},
	year = {2007}}

@article{Fekete,
	author = {Fekete, Michael},
	fjournal = {Mathematische Zeitschrift},
	issn = {0025-5874,1432-1823},
	journal = {Math. Z.},
	mrclass = {99-04},
	mrnumber = {1544613},
	number = {1},
	pages = {228--249},
	title = {{{\"U}}ber die {V}erteilung der {W}urzeln bei gewissen algebraischen {G}leichungen mit ganzzahligen {K}oeffizienten},
	volume = {17},
	year = {1923}}

@article{Fekete-Szego,
	author = {Fekete, Michael and Szeg\H{o}, G{\'a}bor},
	fjournal = {Mathematische Zeitschrift},
	issn = {0025-5874,1432-1823},
	journal = {Math. Z.},
	mrclass = {30.0X},
	mrnumber = {72941},
	mrreviewer = {M.\ Marden},
	pages = {158--172},
	title = {On algebraic equations with integral coefficients whose roots belong to a given point set},
	volume = {63},
	year = {1955}}

@book{Ransford,
	author = {Ransford, Thomas},
	isbn = {0-521-46120-0; 0-521-46654-7},
	mrclass = {31-02},
	mrnumber = {1334766},
	mrreviewer = {D.\ H.\ Armitage},
	pages = {x+232},
	publisher = {Cambridge University Press, Cambridge},
	series = {London Mathematical Society Student Texts},
	title = {Potential theory in the complex plane},
	volume = {28},
	year = {1995}}

@misc{Orloski-Sardari,
	archiveprefix = {arXiv},
	author = {Bryce Joseph Orloski and Naser Talebizadeh Sardari},
	eprint = {2302.02872},
	primaryclass = {math.NT},
	title = {Limiting distributions of conjugate algebraic integers},
	year = {2024}}

@article{Smith,
	author = {Smith, Alexander},
	fjournal = {Annals of Mathematics. Second Series},
	journal = {Ann. of Math.~(2)},
	mrclass = {11R06 (11G10 11G25)},
	mrnumber = {4768417},
	mrreviewer = {James\ McKee},
	number = {1},
	pages = {71--122},
	title = {Algebraic integers with conjugates in a prescribed distribution},
	volume = {200},
	year = {2024}}

@article{Zagier01,
	author = {Zagier, Don},
	fjournal = {Mathematics of Computation},
	journal = {Math. Comp.},
	number = {203},
	pages = {485--491},
	title = {Algebraic numbers close to both {$0$} and {$1$}},
	volume = {61},
	year = {1993}}

@article{Doche2,
	author = {Doche, Christophe},
	fjournal = {Journal de Th\'eorie des Nombres de Bordeaux},
	journal = {J. Th\'eor. Nombres Bordeaux},
	number = {1},
	pages = {103--110},
	title = {Zhang-{Z}agier heights of perturbed polynomials},
	volume = {13},
	year = {2001}}

@article{Doche1,
	author = {Doche, Christophe},
	fjournal = {Mathematics of Computation},
	journal = {Math. Comp.},
	number = {233},
	pages = {419--430},
	title = {On the spectrum of the {Z}hang-{Z}agier height},
	volume = {70},
	year = {2001}}
\bibliographystyle{alphaurl}

%----------------------------------------
\end{document}